\documentclass[12pt]{article}
\usepackage{amsmath}
\usepackage{graphicx,psfrag,epsf}
\usepackage{enumerate}

\newcommand{\blind}{0}

\usepackage{amsmath, amsthm, amssymb, hyperref, soul, color}
\usepackage{graphicx}
\usepackage{caption}
\usepackage{subcaption}
\usepackage{bm}

\usepackage[shortlabels]{enumitem}
\usepackage{graphicx}
\usepackage[all]{xypic}
\usepackage{makecell}
\usepackage{algpseudocode}

\usepackage{lipsum}

\newtheorem{theorem}{Theorem}
\numberwithin{theorem}{section}

\newtheorem{lemma}[theorem]{Lemma}

\newtheorem{definition}[theorem]{Definition}
\newtheorem{problem}[theorem]{Problem}
\newtheorem{remark}[theorem]{Remark}
\newtheorem{example}[theorem]{Example}
\newtheorem{conjecture}[theorem]{Conjecture}
\newtheorem{algorithm}[theorem]{Algorithm}

\newcommand{\RR}{\mathbb{R}}

\newcommand{\Tn}{\mathcal{U}_m}

\newcommand{\tconv}{\text{tconv\,}}
\DeclareMathOperator*{\argmin}{arg\,min}
\newcommand{\Trop}{\text{Trop}}

\begin{document}

\def\spacingset#1{\renewcommand{\baselinestretch}%
{#1}\small\normalsize} \spacingset{1}

%%%%%%%%%%%%%%%%%%%%%%%%%%%%%%%%%%%%%%%%%%%%%%%%%%%%%%%%%%%%%%%%%%%%%%%%%%%%%%

\if0\blind
{
  \title{\bf Tropical principal component analysis on the space of ultrametrics}
  \author{Robert Page\\
    Department of Operations Research, Naval Postgraduate School\\
    and \\
    Ruriko Yoshida\thanks{
    The authors gratefully acknowledge the support of National Science
  Foundation for partially supporting R.Y. (DMS 1916037) and L.Z. (NSF Graduate Research Fellowship). The authors
  also thank Prof.~Bernd Sturmfels for his useful comments.}\hspace{.2cm}\\
    Department of Operations Research, Naval Postgraduate School\\
  and\\
   Leon Zhang\\
   Department of Mathematics, University of California, Berkeley}
  \maketitle
} \fi

\if1\blind
{
  \bigskip
  \bigskip
  \bigskip
  \begin{center}
    {\LARGE\bf  Tropical principal component analysis on the space of ultrametrics}
\end{center}
  \medskip
} \fi

\bigskip
\begin{abstract}
In 2019, Yoshida et al.\ introduced a notion of tropical principal component analysis (PCA).
 The output is a tropical polytope with a fixed number of vertices that best fits the data.
 We here apply tropical PCA to dimension reduction and visualization of data sampled
 from the space of phylogenetic trees. Our main results are twofold:
 the existence of a tropical cell decomposition into regions of fixed tree topology and the development of a stochastic optimization method
 to estimate the tropical PCA using a Markov Chain Monte Carlo (MCMC)
 approach. This method
 performs well with simulation studies, and it is applied to three empirical datasets: Apicomplexa
 and African coelacanth genomes as well as sequences of hemagglutinin for influenza from New York.
\end{abstract}

\noindent%

\spacingset{1.45}
\section{Introduction}
\label{sec:intro}

New technologies allow for the generation of genetic sequences 
cheaply and quickly.  However, it can be challenging to analyze datasets of collections of phylogenetic trees
due to their
high dimensionality and the complex structure of the space of phylogenetic
trees with a fixed number of leaves in which the data lie.

Principal component analysis (PCA) is one of the most popular methods
to reduce dimensionality of input data and to visualize them.  
Classical PCA takes data points in a high-dimensional
Euclidean space and represents them in a lower-dimensional plane in
such a way that the residual sum of squares is minimized.  We cannot directly apply the classical PCA to a set of phylogenetic trees
because the space of phylogenetic trees with a fixed number of leaves is
not Euclidean; it is a union of lower dimensional
polyhedral cones in $\mathbb{R}^{{m \choose 2}}$, where $m$ is the
number of leaves.

Nye showed an algorithm in \cite{Nye} to compute the first order
principal component 
over the space of phylogenetic trees of $m$ leaves.  He defines the
first order principal components as the end points of the unique
shortest connecting paths, or geodesics, defined by the ${\rm
  CAT}(0)$-metric introduced 
by Billera-Holmes-Vogtman (BHV)  over the tree space of phylogenetic
trees with fixed labeled leaves \cite{BHV}. 
Nye in \cite{Nye} used a convex hull of two points, i.e., the geodesic,
on the tree space as the first order PCA.  However, this idea does not generalize to higher order principal components
with the BHV metric as Lin et~al. \cite{LSTY} showed that the
convex hull of three points with the BHV metric over 
the tree space can have arbitrarily high dimension.  

On the other
hand the tropical metric in tree space defined by tropical
  convexity in the max-plus algebra is well-studied and well-behaved
  \cite{MS}. For example, the dimension of the convex hull of $s$
  tropical points is at most $s-1$. In 2019, Yoshida et al.\ in \cite{YZZ} defined a tropical PCA under the tropical
metric  with the max-plus tropical arithmetic in two ways: the
best-fit Stiefel
tropical linear space of fixed dimension closest to the data points in
the tropical projective torus, and the best-fit
tropical polytope with a fixed number of vertices closest to the data
points.  The authors showed that the latter object can be written as a
mixed-integer programming problem to compute them, and they applied the second definition to
datasets consisting of collections of phylogenetic trees. Nevertheless, exactly computing the best-fit tropical polytope can be expensive due to the high-dimensionality of the mixed-integer programming problem.

This paper focuses on the same approach to tropical PCA as a tropical polytope over the
space of equidistant trees with a fixed set of leaves.  In order to
use tropical PCA to visualize a dataset of phylogenetic trees, we show the existence of a decomposition of the tropical PCA into regions of unchanging tree topology deriving intrinsically from tropical geometry in Section \ref{sec:PCA-props}. 
% Hower, in order to apply their methods to a set of phylogenetic trees,
% they have only one approach to estimate the tropical PCA by
% enumerating all possible $s$-subset of trees from the input dataset
% for computing the $(s-1)$th order principal compoenents, because they
% defined a tropical PCA over the tropical projective torus not over a
% space of phylogenetic trees, that is, a lower dimensional subspace of
% the the tropical projective torus.  Therefore, at this moment, there is
% no method to compute a tropical PCA over the space of phylogenetic
% trees.  
In Section \ref{tropPCA} we propose a heuristic method to compute a tropical PCA as a
tropical polytope over a
space of rooted phylogenetic trees with a fixed number of leaves.  % \hl{More
% specifically we propose a stochastic computation using Metropolis
% Hasting algorithm to estimate a tropical PCA over the space of \em
%   equidistant phylogenetic trees with the number of leaves $m$. } 
To show its performance, we conduct intensive simulation studies with
the proposed method in Section \ref{sec:simulations}.  We end this paper in Section \ref{sec:experiments} with computational experiments
on three datasets consisting of Apicomplexa and African
coelacanth genomes as well as sequences of hemagglutinin for influenza. 

% This paper is organized as follows:  In Section \ref{basics} we will
% discuss basics of tropical arithmetics under the {\em
%   max-plus} algebra which we will use throughout this paper.  In
% Section \ref{equid} we will review some definitions from phylogenetics
% and we will connect them to space of {\em ultrametrics} as a tropical
% linear space.  
% In Section \ref{tropPCA} we will discuss a tropical PCA defined as a
% tropical polytope which introduced by \cite{YZZ}.  They interpret a
% tropical PCA as a best-fit tropical polytope to an input data.  Then
% in Section \ref{MCMC} we will apply Metropolis-Hasting algorithm to
% estimate the best-fit tropical polytope closest to a sample of
% equidistant trees with $m$ leaves, over the space of ultrametrics.
% Then we will conduct simulation studies to show our approximation
% works very well in general and we will end this paper with
% applications to two empirical datasets, Apicomplexa and African
% coelacanth genomes. 

\section{Tropical basics}\label{basics}
In this section we review some basics of tropical arithmetic and tropical
geometry pertaining to our setting.  See
\cite{MS} or \cite{J} for more detail. 

\begin{definition}[Tropical arithmetic operations]
Throughout this paper we will perform arithmetic in the
max-plus tropical semiring $(\,\mathbb{R} \cup \{-\infty\},\oplus,\odot)\,$.
In this tropical semiring,  we define the basic tropical
arithmetic operations of addition and multiplication as:
$$a \oplus b := \max\{a, b\}, ~~~~ a \odot b := a + b ~~~~\mbox{  where } a, b \in \mathbb{R}\cup\{-\infty\}.$$
Note that $-\infty$ is the identity element under addition and 0 is the identity element under multiplication.% : for all $a\in \mathbb{R}\cup \{-\infty\}$, we have $a \oplus -\infty = a$ and $a \odot 0 = a$.
\end{definition}

\begin{definition}[Tropical scalar multiplication and vector addition]
For any scalars $a,b \in \mathbb{R}\cup \{-\infty\}$ and for any vectors $v = (v_1,
\ldots ,v_e), w= (w_1, \ldots , w_e) \in (\mathbb{R}\cup-\{\infty\})^e$, we 
define tropical scalar multiplication and tropical vector addition as follows:
$$a \odot v = (a + v_1, a + v_2, \ldots ,a + v_e)$$
$$a \odot v \oplus b \odot w = (\max\{a+v_1,b+w_1\}, \ldots, \max\{a+v_e,b+w_e\}).$$
\end{definition}

Throughout this paper we consider the \emph{tropical projective
  torus} $\mathbb R^e \!/\mathbb R {\bf
  1}$, where ${\bf 1}:=(1, 1, \ldots , 1)$ is the all-ones vector.  

\begin{definition}[Generalized Hilbert projective metric]
For any two points $v, \, w \in \mathbb R^e \!/\mathbb R {\bf
  1}$,  the {\em tropical
  distance} $d_{\rm tr}(v,w)$ between $v= (v_1, \ldots , v_e)$ and $w = (w_1,\dots, w_e)$ is defined as:
\begin{equation}
\label{eq:tropmetric} d_{\rm tr}(v,w) \,\, = \,\,
\max_{i,j} \bigl\{\, |v_i - w_i  - v_j + w_j| \,\,:\,\, 1 \leq i < j \leq e \,\bigr\}.
\end{equation}This
distance measure is a metric in $\mathbb R^e \!/\mathbb R {\bf
  1}$ and it is also known as the
{\em generalized Hilbert projective metric} 
\cite[\S 2.2]{AGNS}, \cite[\S 3.3]{CGQ}.
  \end{definition}

\begin{definition}
A subset $S \subset \mathbb{R}^e$ is called \textit{tropically convex} if it contains the point $a \odot  x \oplus b \odot y$ for any $x,y
\in S$ and any $a, b \in \mathbb{R}$. The \textit{tropical convex
  hull} or \textit{tropical polytope} ${\rm tconv}(V)$ of a finite subset $V \subset
\mathbb{R} ^e$ is the smallest tropically convex subset containing $V
\subset \mathbb{R}^e$. The tropical convex hull of $V$ can be 
written as  the set of all tropical linear combinations 
$${\rm tconv}(V) = \{a_1 \odot v_1 \oplus a_2 \odot v_2 \oplus \cdots \oplus a_r \odot v_r : v_1,\ldots,v_r \in V \mbox{ and } a_1,\ldots,a_r \in \mathbb{R}\}.$$
Any tropically convex subset $S$ of $\mathbb{R}^e$ is closed under tropical
scalar multiplication, $\mathbb{R} \odot S \subseteq S$. 

Finally, a subset of $\mathbb R^e/\mathbb R{\bf 1}$ is said to be \emph{tropically convex} if it is the quotient of a tropically convex subset of $\mathbb R^e$.
\end{definition}

\begin{definition}
Consider a tropical polytope
$\mathcal{P} = {\rm tconv}(D^{(1)},D^{(2)},\ldots,D^{(s)})$.
The $D^{(i)}$ are points in $\RR^e/\RR {\bf 1}$.
We have
\begin{equation}
\label{eq:tropproj} %\label{eq:nearestpoint}
 \pi_\mathcal{P} (D) \,= \,
\lambda_1 \odot  D^{(1)} \,\oplus \,
\lambda_2 \odot  D^{(2)} \,\oplus \, \cdots \,\oplus \,
\lambda_s \odot  D^{(s)}  ,
\quad {\rm where} \,\, \lambda_k = {\rm min}(D-D^{(k)}) .
\end{equation}
This formula appears in \cite[(5.2.3)]{MS}. It allows us to easily
project an ultrametric $D$ (or any other point in $\RR^e$) onto the
tropical convex hull of $s$ given ultrametrics.
\end{definition}

\begin{definition}
\label{def:types}
Let $\mathcal P = \tconv(D^{(1)}, \dots, D^{(s)})\subseteq \mathbb R^{e}/\RR{\bf 1}$ be a tropical polytope. Each point $x$ in $\mathbb R^{ e}/\RR{\bf 1}$ has a \emph{type} $S = (S_1,\dots, S_{ e })$ according to $\mathcal P$, where an index $i$ is in $S_j$ if
\[D^{(i)}_j - x_j = \max(D^{(i)}_1-x_1,\dots, D^{(i)}_{ e }-x_{ e }).\]
The tropical polytope $\mathcal P$ consists of all points $x$ whose type $S = (S_1,\dots, S_{e})$ has all $S_i$ nonempty. Each collection of points with the same type is called a \emph{cell}.
\end{definition}

\begin{example}
\label{example-polytopes}
Consider the five points $D^{(1)} = (0,0,0), D^{(2)} = (0, 3, 0), D^{(3)} = (0,3,3), D^{(4)} = (0,1,2)$, and $D^{(5)} = (0,2,1)$. The tropical convex hull of the first four points and of all five points are presented below, along with the decomposition of the polytope into cells.
\begin{figure}[h!]
\begin{center}
\includegraphics[height=2in]{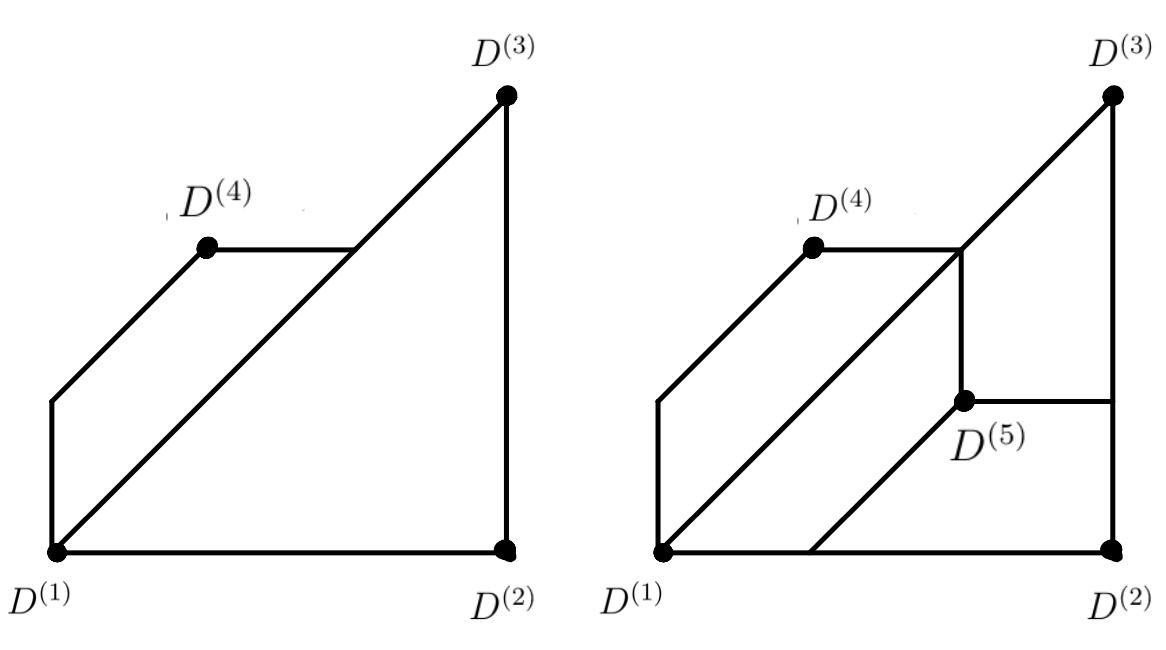}
\end{center}
\caption{The two tropical polytopes of Example
  \ref{example-polytopes}.  Note that we visualize the tropical projective torus $\mathbb R^e/\mathbb R{\bf 1}$ by forcing the first coordinate of each point to be 0.}
\end{figure}
\end{example}

\subsection{Space of ultrametrics}\label{equid}

Next we describe some basics on phylogenetic trees and review an interpretation of the space of equidistant trees as a tropical linear space. 

\begin{definition}
Let $T=(V,E)$ be a tree with leaves with labels $[m]: = \{1, \ldots ,
m\}$ but no labels on internal nodes in $T$.   We call such a tree a \emph{phylogenetic tree}.
An {\em equidistant tree} is a rooted phylogenetic tree such that a total
branch length from its root to each leaf is the same. 
\end{definition}

\begin{definition}
A \emph{dissimilarity map} $d$ is a function $d:[m]\times [m]\to
\mathbb R_{\geq 0}$ such that $d(i,i)=0$ and $d(i,j)=d(j,i)\geq 0$ for
each $i,j\in [m]$. If a dissimilarity map $d$ satisfies a triangle
inequality, i.e., $d(i,j)\leq
d(i,k)+d(k,j)$ for all $i,j,k\in [m]$, then we call $d$ a
\emph{metric}. 
We can represent a dissimilarity map $d$ by an $m\times m$ matrix $D$
whose $(i,j)$th entry is $d_{ij}$. Because $D$ is clearly symmetric
and all diagonal entries are zeros, we can regard
$d$ as a vector in $\mathbb R^{e} = \mathbb R^{\binom m 2}$. 
\end{definition}
Note that if it is clear, for convenience, we write $d_{ij}$ for $d(i,j)$. 

\begin{definition}
Let $T$ be a phylogenetic tree with $m$ leaves labeled with the
elements of $[m]$, and assign a length $\ell_e\in \mathbb R$ to each
edge $e$ of $T$. Define $d:[m]\times [m]\to \mathbb R$ such that
$d_{ij}$ is the total length of the unique path from leaf $i$ to leaf
$j$. We call a function $d$ obtained in this way a \emph{tree
  distance}. If, furthermore, each entry of the distance matrix $D$ is
nonnegative, then $d$ is in fact a metric. We call such a tree distance a
\emph{tree metric}. As before, we can embed $D$ into $\mathbb R^{e}$.  
\end{definition}

\begin{definition}
Let $d:[m]\times [m]\to \mathbb R_{\geq 0}$ be a metric which satisfies the following strengthening of the triangle inequality for each choice of $i,j,k\in [m]$:
\[d(i,k)\leq \max(d(i,j),d(j,k)).\]
We call such a metric an \emph{ultrametric}. Let $\mathcal U_m$ denote the collection of all ultrametrics in $\mathbb R^e/{\mathbb R \bf 1}$.
\end{definition} 

\begin{remark}
It is well-known that a distance matrix is a tree metric for an
equidistant tree if and only if it is ultrametric.  
\end{remark}

Therefore, we view $\mathcal U_m$ as the space of
equidistant trees with fixed leaf labels.  
For the expert in tropical geometry the following theorem shows an
explicit description of the space of equidistant trees with fixed
labels of leaves. 

\begin{theorem}[\cite{DS}]
Let $L_m$ be the subspace of $\mathbb R^e$ defined by the linear
equations $x_{ij} - x_{ik} + x_{jk}=0$ for $1\leq i < j <k \leq
m$. Let $\Trop(L_m)\subseteq \RR^e/\RR {\bf 1}$ be the tropicalization
of the linear space with points $(v_{12},v_{13},\ldots, v_{m-1,m})$
such that $\max(v_{ij},v_{ik},v_{jk})$ is obtained at least twice for
all triples $i,j,k\in [m]$.

Then the image of $\mathcal U_m$ in the tropical projective torus
$\RR^e/\RR {\bf 1}$ coincides with $\Trop(L_m)$. 
\end{theorem}

Therefore, the space of equidistant trees with fixed labels of leaves
is a tropical linear space defined by tropical equations.  Figure
\ref{fig:U3} shows a space of ultrametrics for $m = 3$.  This is a one-dimensional
tropical linear space. 

\begin{example}
For $m = 3$, the space of ultrametrics $\mathcal{U}_3$ is shown in
Figure \ref{fig:U3}.   This is a  tropical linear space defined by the tropical condition that $\max(v_{12}, v_{13}, v_{23})$ is attained twice.
\begin{figure}
  \centering
        \includegraphics[width=0.5\textwidth]{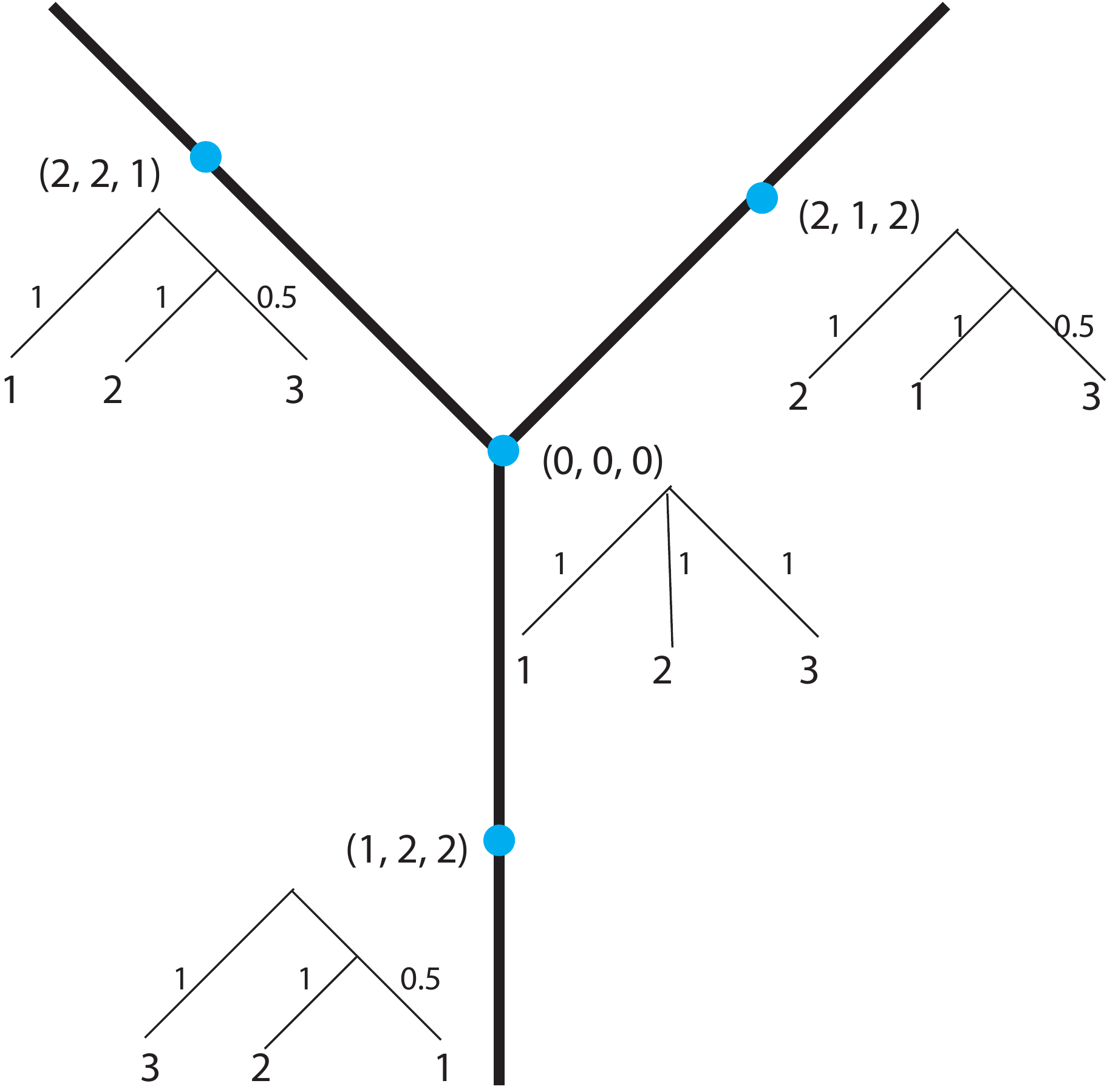}
  \caption{The space of ultrametrics for $m = 3$.  Blue dots
    correspond to the equidistant trees shown in the figure. }\label{fig:U3}
\end{figure}
\end{example}

\section{Properties of tropical PCA}
\label{sec:PCA-props}
We finally recall our notion of tropical PCA.  It is important that we
can interpret the tropical PCA in terms of equidistant trees.
Therefore, we seek to prove properties about its interpretation.

\begin{definition}
Let $\mathcal P = \tconv(D^{(1)}, \dots, D^{(s)})\subseteq \mathbb
R^{e}/\RR{\bf 1}$ be a tropical polytope with its vertices
$\{D^{(1)},\dots, D^{(s)}\} \subset \mathbb{R}^e/{\mathbb R} {\bf 1}$ and let $S = \{u_1,
\ldots u_n\}$ be a sample from the space of
ultrametrics $\mathcal{U}_m$.
Let $\Pi_{\mathcal P}(S):= \sum_{i = 1}^{|S|}d_{tr}(u_i, u'_i)$, where 
$u'_i$ is the tropical projection of $u_i$ onto a tropical
polytope $\mathcal P$.
Then the vertices $D^{(1)},\dots, D^{(s)}$ of the tropical
polytope $\mathcal P$ are called the $(s-1)$-th order tropical principal 
components of $S$ if the tropical polytope $\mathcal P$ minimizes
$\Pi_{\mathcal P}(S)$ over all possible tropical polytopes with $s$
many vertices.
% Suppose we have a sample $\{D^{(1)},\dots, D^{(s)}\}$. As in \cite{YZZ}, we define a $(s-1)$th order
% tropical PCA as a tropical convex hull of $s$ points in the space of
% ultrametrics $\Tn$ minimizing the sum of distances to the data points.
\end{definition}

One nice property of our tropical PCA is that each cell comprises
ultrametrics of the same tree topology.

\begin{theorem}
Let $\mathcal P = \tconv(D^{(1)}, \dots, D^{(s)})\subseteq \mathbb R^{e}/\RR{\bf 1}$ be a tropical polytope spanned by ultrametrics. Then any two points $x$ and $y$ in the same cell of $\mathcal P$ are also ultrametrics with the same tree topology.
\end{theorem}
\begin{proof}
Because the space of ultrametrics $\mathcal U_m$ is a tropical linear space, so is tropically convex, $\mathcal P$ is contained in $\mathcal U_m$. Hence any points $x$ and $y$ in $\mathcal P$ must also be ultrametrics.

Let $S$ be the type of $x$ and $y$. To check whether $x$ and $y$ have the same tree topology, we check the three point condition for each trio of leaves. Fix such a trio $i,j,$ and $k$. Our first claim is that $x_{ij} = x_{ik} = x_{jk}$ if and only if $y_{ij} = y_{ik} = y_{jk}$. To see why, suppose the former is true. Because $x\in\mathcal P$, there exists some index $a\in S_{ij}$, so that
\[D^{(a)}_{ij}-x_{ij}=\max_{\ell_1< \ell_2}D^{(a)}_{\ell_1,\ell_2}-x_{\ell_1,\ell_2}.\]
In particular, $D^{(a)}_{ij}-x_{ij}\ge D^{(a)}_{ik}-x_{ik}$ and $D^{(a)}_{ij}-x_{ij}\ge D^{(a)}_{jk}-x_{jk}$. By assumption it follows that $D^{(a)}_{ij}\ge D^{(a)}_{ik}$ and $D^{(a)}_{ij}\ge D^{(a)}_{ik}.$ Because $D^{(a)}$ is an ultrametric, the maximum among $D^{(a)}_{ij}, D^{(a)}_{ik}, D^{(a)}_{jk}$ is attained twice, so one of these is actually an equality. Without loss of generality, let $D^{(a)}_{ij}=D^{(a)}_{ik}$. Then $D^{(a)}_{ij}-x_{ij}=D^{(a)}_{ik}-x_{ik}$ and $a\in S_{ik}$ as well.

Recall that $S$ is also the type of $y$. This means 
\[D^{(a)}_{ij}-y_{ij}=D^{(a)}_{ik}-y_{ik}=\max_{\ell_1< \ell_2}D^{(a)}_{\ell_1,\ell_2}-y_{\ell_1,\ell_2}.\]
In particular, since $D^{(a)}_{ij}=D^{(a)}_{ik}$, we have $y_{ij}=y_{ik}$ as well. The same argument applied to $S_{jk}$ shows that $y_{jk}=y_{ij}$ or $y_{jk}=y_{ik}$; it follows that $y_{ij}=y_{ik}=y_{jk}$ as desired.

Now suppose $\max(x_{ij},x_{ik},x_{jk})$ and $\max(y_{ij},y_{ik},y_{jk})$ are both attained exactly twice. We claim that the minimum for $x$ and for $y$ is attained for the same pair of leaves. Suppose without loss of generality that $x_{ij}=\min(x_{ij},x_{ik},x_{jk})$ and $y_{ik}=\min(y_{ij},y_{ik},y_{jk})$: because $x,y\in\mathcal P$, there exists some index $a\in S_{jk}$. This implies in particular that $D^{(a)}_{jk}-x_{jk}\ge D^{(a)}_{ij}-x_{ij}$ and $D^{(a)}_{jk}-y_{jk}\ge D^{(a)}_{ik}-y_{ik}$. 

Rearranging these inequalities produces $D^{(a)}_{jk}-D^{(a)}_{ij}\ge x_{jk}-x_{ij}$ and $D^{(a)}_{jk}-D^{(a)}_{ik}\ge y_{jk}-y_{ik}$. Since $x_{ij}<x_{jk}$ and $y_{ik}<y_{jk}$ by assumption, it follows that $D^{(a)}_{jk}$ must be the unique maximum among $D^{(a)}_{ij}, D^{(a)}_{ik}, D^{(a)}_{jk}$, a contradiction because $D^{(a)}$ is ultrametric.
\end{proof}
See Figures \ref{fig:lung} and \ref{fig:apicomplexa} for illustrations of this result. 

It is natural to ask when a tropical PCA contains the origin: i.e., when the fully unresolved phylogenetic tree is contained in the PCA. This question turns out to have a simple answer.
\begin{lemma}
Let $\mathcal P = \tconv(D^{(1)}, \dots, D^{(s)})\subseteq \mathbb R^{e}/\RR{\bf 1}$ be a tropical polytope spanned by ultrametrics. The origin $\bf 0$ is contained in $\mathcal P$ if and only if the path between each pair of leaves $i,j$ passes through the root of some $D^{(i)}$. 
\end{lemma}
\begin{proof}
The $D^{(i)}$ can certainly be tropically scaled to have largest coordinate 0. If the claimed condition holds, then the sum of these scaled $D^{(i)}$ will be the origin as desired.

Suppose $\bf 0$ is contained in $\mathcal P$, meaning we can write ${\bf 0} =\bigoplus a_i \odot D^{(i)}$. Consider some pair of leaves $i,j$. This must appear as the coordinate of some $a_k\odot D^{(k)}$. Because ${\bf 0}=\bigoplus a_i \odot D^{(i)}$, all other coordinates of $a_k\odot D^{(k)}$ must be non-positive, meaning that the $i,j$ coordinate of $a_k\odot D^{(k)}$ is maximal as desired.
%We must be able to write , \cite{DS}[Theorem 23] states that $\mathcal P$ is isomorphic to another tropical polytope $\mathcal P'$ spanned by $e$ vertices in $\mathbb R^{s}/\mathbb R{\bf 1}$. The map from $\mathcal P\to\mathcal P'$ sends $x$ to $(\max(D^{(1)}-x), \dots, \max(D^{(s)}-x)).$ This means $0\mapsto (2,2,\dots, 2)=0$ because $\max(D^{(i)})=2$ for each spanning point $D^{(i)}.$ The map $\mathcal P'\to \mathcal P$ sends $y\in \mathcal P'$ to the point whose $i$th coordinate is $\max(D^{(1)}_i-y_1, \dots, D^{(s)}_i-y_s)$. Hence $0$ is sent by the composition of maps $\mathcal P\to\mathcal P'\to\mathcal P$ as $0\mapsto0\mapsto D^{(1)}\oplus\cdots\oplus D^{(s)}=0$ because these two maps are inverses.
\end{proof}

\begin{definition}
Suppose we have a sample $\{D^{(1)}, \dots, D^{(n)}\}$. A {\em
  Fermat-Weber point} $x^*$ of \\
$\{D^{(1)}, \dots, D^{(n)}\}$ is a minimizer of the sum of tropical distances to the data points:
\[
x^* := \argmin_x \sum_{i=1}^n d_{\rm tr} (x, D^{(i)}).
\]
We can naturally view Fermat-Weber points as zero-dimensional PCAs.
\end{definition}

\begin{lemma}
Suppose $n \geq 3$ and $ \{D^{(1)}, \ldots ,  D^{(n)} \} \subset \mathcal{U}_m$.
Then there exists a 
Fermat-Weber point $x^*$ of the dataset lying in the tropical polytope $\tconv (D^{(1)},\dots, D^{(n)})$. In particular, this point $x^*$ is ultrametric.
\end{lemma}
\begin{proof}
Take $x^*$ a Fermat-Weber point not lying in $\tconv (D^{(1)},\dots, D^{(n)})$, and let $S = (S_1,\dots, S_e)$ be its vector of types as in Definition \ref{def:types}. Because $x^*$ does not lie in the tropical polytope, some of the types $S_j$ are empty. Consider such an $S_j$. By definition, we have that for each $i$, $D^{(i)}_j-x^*_j$ is not maximal among $\{D^{(i)}_1-x_1^*,\dots, D^{(i)}_e-x_e^*\}$. This also means that we can shift the $j$th coordinate without changing the distance of $x^*$ to any datapoint $D^{(i)}$.  We can therefore simply decrease $x^*_e$ until there is some $i$ such that $D_j^{(i)}-x_j^*$ is tied for being maximal among the coordinates of $D^{(i)}-x^*$. The tropical type $S_j$ of our new $x^*$ will be nonempty. By doing so for all coordinates, we obtain a new Fermat-Weber point which lies in the tropical polytope spanned by ultrametrics and so is itself ultrametric.
\end{proof}

The previous lemma states that there always exists a biologically interpretable zero-dimensional tropical PCA for a dataset of ultrametrics. This result points toward the following conjecture, which is analogous to the classical fact that the $s$-dimensional PCA is contained in the $t$-dimensional PCA if $s\le t$.  
\begin{conjecture}
There exists a tropical Fermat-Weber point $x^*\in\mathcal U_m$ of a
sample $D^{(1)}, \dots, D^{(n)}$ of ultrametric trees which is
contained in the $s$th order tropical PCA of the dataset for $s
\geq 1$.
\end{conjecture}

\section{Methods}
\label{sec:meth}\label{tropPCA}\label{MCMC}

In this section we discuss how to estimate the optimal
solution to a tropical PCA via a Markov Chain Monte Carlo (MCMC).  This can be applied to estimate the $(s-1)$-th order principal components
for $s \geq 1$.

In the remainder of this paper we often focus on $s = 3$ for
simplicity, even though our techniques apply for any $s \geq 3$.
Finding the 3rd order tropical PCA can be written as the following
optimization problem:
\begin{problem}\label{optimization}
We seek a solution for the following optimization problem:
\[%\begin{equation}\label{cost:convex}
\min_{D^{(1)}, D^{(2)}, D^{(3)} \in \Tn} \sum_{i = 1}^n  d_{\rm
  tr}(u_i, u'_i)
\]%\end{equation}
where
\begin{equation}\label{const1:convex}
u'_i= \lambda_1^i \odot  D^{(1)} \,\oplus \,
\lambda_2^i \odot  D^{(2)} \,\oplus \,
\lambda_3^i \odot  D^{(3)}  ,
\quad {\rm where} \,\, \lambda_k^i = {\rm min}(u_i-D^{(k)}) ,
\end{equation}
and 
\begin{equation}\label{const2:convex}
d_{\rm
  tr}(u_i, u'_i) = \max\{|u_i(k) - u'_i(k) - u_i(l) + u'_i(l)|: 1 \leq
k < l \leq e\}
\end{equation}
with 
\begin{equation}\label{const3:convex}
u_i = (u_i(1), \ldots , u_i(e)) \text{ and } u'_i = (u'_i(1), \ldots , u'_i(e)).
\end{equation}
\end{problem}
 \bigskip 

%\section{Markov Chain Monte Carlo for estimation}\label{MCMC}

Let $\Pi_{\Delta}(S):= \sum_{i = 1}^{|S|}d_{tr}(u_i, u'_i)$, where $S = \{u_1,
\ldots u_n\}$ with $u_i \in \RR^e/\RR {\bf 1}$ are ultrametrics and
$u'_i$ is the tropical projection of $u_i$ onto a tropical
triangle $\Delta$.

The following algorithm computes a proposal state, i.e., a set of
proposed trees.  
\begin{algorithm}[Finding the proposal set of trees]\label{MCMC1}  \qquad
\begin{itemize}
\item Input: Set of equidistant trees $\{T_1, T_2, T_3\}$, $k \in [m]$.
\item Output: Next set of equidistant trees $\{T'_1, T'_2, T'_3\}$.
\end{itemize}
\begin{algorithmic}
\For{\texttt{$i = 1, \ldots , 3$}} 
	\State  Set $T'_i = T_i$.
	\State  Pick random numbers $(i_1, \ldots , i_k) \subset [m]$
without replacement.
	\State Permute the tree leaf labels $(i_1, \ldots , i_k) \subset
        [m]$ of $T'_i$ with a random permutation $\sigma$ in the
        symmetric group on $\{i_1, \ldots , i_k\}$.
	\State Pick a random internal branch $b_1$ in $T'_i$ with  branch length $l_i$.
	\State Update $l_i := l_i + \epsilon \cdot c$ where $\epsilon
        \sim Unif\{\pm 1\}$, and $c \sim Unif[0, l_i/m]$.
        \State Pick another branch $b_2$ with branch length $l$ on the path from the root to a leaf where the branch $b_1$ is also on the path. 
       \State If $l-\epsilon \cdot c < 0$ then set $l := 0$ and $l_i := l_i + l - \epsilon
       \cdot c$.  If not then set $l:= l - \epsilon \cdot c$.
\EndFor
\State Return $\{T'_1, T'_2, T'_3\}$.
\end{algorithmic}
\end{algorithm}

Next we use the Metropolis algorithm to decide whether the proposal state
should be accepted or rejected.  Let $\Delta_{(w_1, w_2, w_3)}$ be the tropical triangle spanned
by $w_1, w_2, w_3$.

\begin{algorithm}[Metropolis algorithm]\label{MCMC2}  \qquad
\begin{itemize}
\item Input: Current set of equidistant trees $\{T_1, T_2, T_3\}$ and
  the proposal state, $\{T'_1, T'_2, T'_3\}$.  The sample of
  ultrametrics $S = \{u_1, \ldots u_n\}$.
\item Output: Decision whether we should accept the proposal or not.  
\end{itemize}
\begin{algorithmic}
\State Compute ultrametrics $w_1, w_2, w_3$, from $T_1, T_2, T_3$,
respectively.
\State Compute ultrametrics $v_1, v_2, v_3$, from $T'_1, T'_2, T'_3$,
respectively.
\State Compute $\Pi_{\Delta_{w_1, w_2, w_3}}(S)$ and
$\Pi_{\Delta_{v_1, v_2, v_3}}(S)$.
\State Set $p = \min\{1, 
\Pi_{\Delta_{w_1, w_2, w_3}}(S)/ \Pi_{\Delta_{v_1, v_2, v_3}}(S)\}$.
\State Accept a proposal $\{T'_1, T'_2, T'_3\}$ with probability $p$.
\end{algorithmic}
\end{algorithm}

Piecing together Algorithms \ref{MCMC1} and \ref{MCMC2}, we have the following
MCMC algorithm.

\begin{algorithm}[MCMC algorithm to estimate the second order
  principal components]\label{MCMC3}  \qquad
\begin{itemize}
\item Input: Sample of equidistant trees $\{T_1, \ldots , T_n\}$.
  Constant positive integer $C > 0$.
\item Output: Second order principal components $\{T^*_1, T^*_2, T^*_3\}$.
\end{itemize}
\begin{algorithmic}
\State Set $S := \{u_1, \ldots , u_n\}$ where $d_i$ is the
ultrametrics computed from a tree $T_i$, for $i = 1, \ldots , n$.
\State Pick random trees $\{T_0^1, T_0^2, T_0^3\} \subset \{T_1,
\ldots , T_n\}$ and compute ultrametrics $w^*_1, w^*_2, w^*_3$ respectively.
\State Set $i = 1$, $k = m$, where $m$ is the number of leaves.
\Repeat
       \If{\texttt{$i$ mod $C$ equals zero and $k > 0$}}
       \State  set $k = k - 1$.
       \EndIf
       \State Compute the proposal $\{T_1^1, T_1^2, T_1^3\}$ via
       Algorithm \ref{MCMC1} with $\{T_0^1, T_0^2, T_0^3\}$ and $k$.
      \State Compute ultrametrics $w_1, w_2, w_3$, from $T_1^1, T_1^2, T_1^3$,
respectively.
       \If{\texttt{Algorithm \ref{MCMC2} returns ``accept''}}  
       \State Set $T_0^1 = T_1^1$, $T_0^2 = T_1^2$, and $T_0^3 = T_1^3$.
       \EndIf
       \If{\texttt{$\Pi_{\Delta_{w_1, w_2, w_3}}(S) < \Pi_{\Delta_{w^*_1, w^*_2,
             w^*_3}}(S)$}} 
        \State set $w^*_1:= w_1, \, w^*_2 := w_2, \, w^*_3:= w_3$.
      \EndIf
       \State Set $i = i + 1$.
\Until{\texttt{Converges}}
\State Return the ultrametrics $w^*_1, w^*_2, w^*_3$.
\end{algorithmic}
\end{algorithm}

\subsection{Fraction of variance of unexplained and variance of explained}

To analyze the fit of a tropical PCA to the observed data, we used
a fraction of variance of unexplained $\Pi_{\Delta}(S)$ and variance of explained
in terms of our tropical geometric set up.  In tropical geometry, it is
natural to use a Fermat-Weber point as a centroid of the given
datasets \cite{LSTY}, and correspondingly we will use a sum of tropical distances instead of the
squared distances.  
Let $S_{reg}$ be the ``variance of explained'',  defined as 
\[
S_{reg} = \sum_{i=1}^n d_{tr}(\hat{u}_i, \bar{u})
\]
where $\hat{u}_i$ is the tropical projection of an ultrametric $u_i$
for a tree $T_i$ in the input sample onto a tropical polytope and
$\bar{u}$ is a Fermat Weber point of $\{\hat{u}_i, \ldots ,
\hat{u}_n\}$ as defined in Section \ref{sec:PCA-props}.

We define the fraction of
variance of unexplained as 
\[
\frac{\Pi_{\Delta}(S)}{\Pi_{\Delta}(S) + S_{reg}}.
\]

Here the coefficient of determination (or the proportion of the
variance of explained), or $R^2$, is defined as
\begin{equation}\label{r2def}
R^2 = 1 - \frac{\Pi_{\Delta}(S)}{\Pi_{\Delta}(S) + S_{reg}} = \frac{S_{reg}}{\Pi_{\Delta}(S) + S_{reg}}.
\end{equation}
We use the proportion of determination $R^2$ as the statistic to
measure how well the model fits to the given data.  
% \hl{which notation}
% \hl{which one is explained and which one is unexplained?}

\section{Verifications}
\label{sec:verify}\label{sec:simulations}

\subsection{Mixture of coalescent models}

For the first simulation study, we generated gene trees with a species tree
under a coalescent model via the software {\tt Mesquite} \cite{mesquite}.  We fixed
the effective population size $N_e = 100,000$ and varied
\[
r = \frac{SD}{N_e}
\]
where $SD$ is the species depth.

\begin{algorithm}[Sample gene trees from a coalescent model]\label{al1}  \qquad
\begin{itemize}
\item Input: The number of gene trees $n$, labels of leaves $\{1,
  \ldots , m\}$, effective population size $N_e$ and species depth $SD$.
\item Output: A sample of $n$ gene trees with a fixed species
  tree $T_s$ under a coalescent model.
\end{itemize}
\begin{algorithmic}
\State Generate a species tree $T_s$ with the label  $\{1,
  \ldots , m\}$ and $N_e$ under Yule model. 
\State Generate $n$ many gene trees with $T_s$ and $N_e$ under the
coalescent model via {\tt Mesquite}.
\State Return the gene trees generated and $T_s$.
\end{algorithmic}
\end{algorithm}

%\subsection{Simulation Experiments}

%Here we conduct two sets of experiments.  
This experiment is to
generate two distributions of gene trees under the coalescent model with
different species trees using Algorithm \ref{al1}. Then we use
Algorithm \ref{al2} to compute a tropical PCA on the mixture of these
two distributions generated.   In these simulated
experiments, we have varied the ratio $r = 0.25, 0.5, \, 
1, \, 2, \, 5, \, 10$.  We also fix the number of leaves $m = 10$.  

\begin{algorithm}[PCA with two distributions of gene trees with
  difference species trees]\label{al2}  \qquad
\begin{itemize}
\item Input:  $S_1 := \{T_1, \ldots , T_n\}$, a sample of gene trees
  generated with a species tree $T_{S_1}$ and $S_2 := \{T'_1, \ldots ,
  T'_n\}$, a sample of gene trees 
  generated with a species tree $T_{S_2}$ where $T_{S_1} \not = T_{S_2}$.
\item Output: PCA (second PCs) with a sample $\{T_1, \ldots , T_n,
  T'_1, \ldots , T'_n\}$.
\end{itemize}
\begin{algorithmic}
\State Apply a tropical PCA and compute the second order principal
components with $\{T_1, \ldots , T_n,
  T'_1, \ldots , T'_n\}$. 
\State Color red for the projected points of $\{T_1, \ldots , T_n\}$
onto the tropical PCA and color blue for for the projected points of
$\{T'_1, \ldots , T'_n\}$ 
onto the tropical PCA.
\end{algorithmic}
\end{algorithm}

In order to compare our results, we also applied the same simulated
data sets to {\tt geophytter} which approximate the BHV PCA on the
BHV tree space \cite{NTWY}.
To compare the accuracy rates we use $R^2$, the variance of
explained. Note that the $R^2$ statistic for the BHV PCA is defined in an analogous way to our $R^2$, simply replacing the tropical distance with the BHV distance.

\begin{figure}
    \centering
    \begin{subfigure}[b]{0.3\textwidth}
        \includegraphics[width=0.9\textwidth]{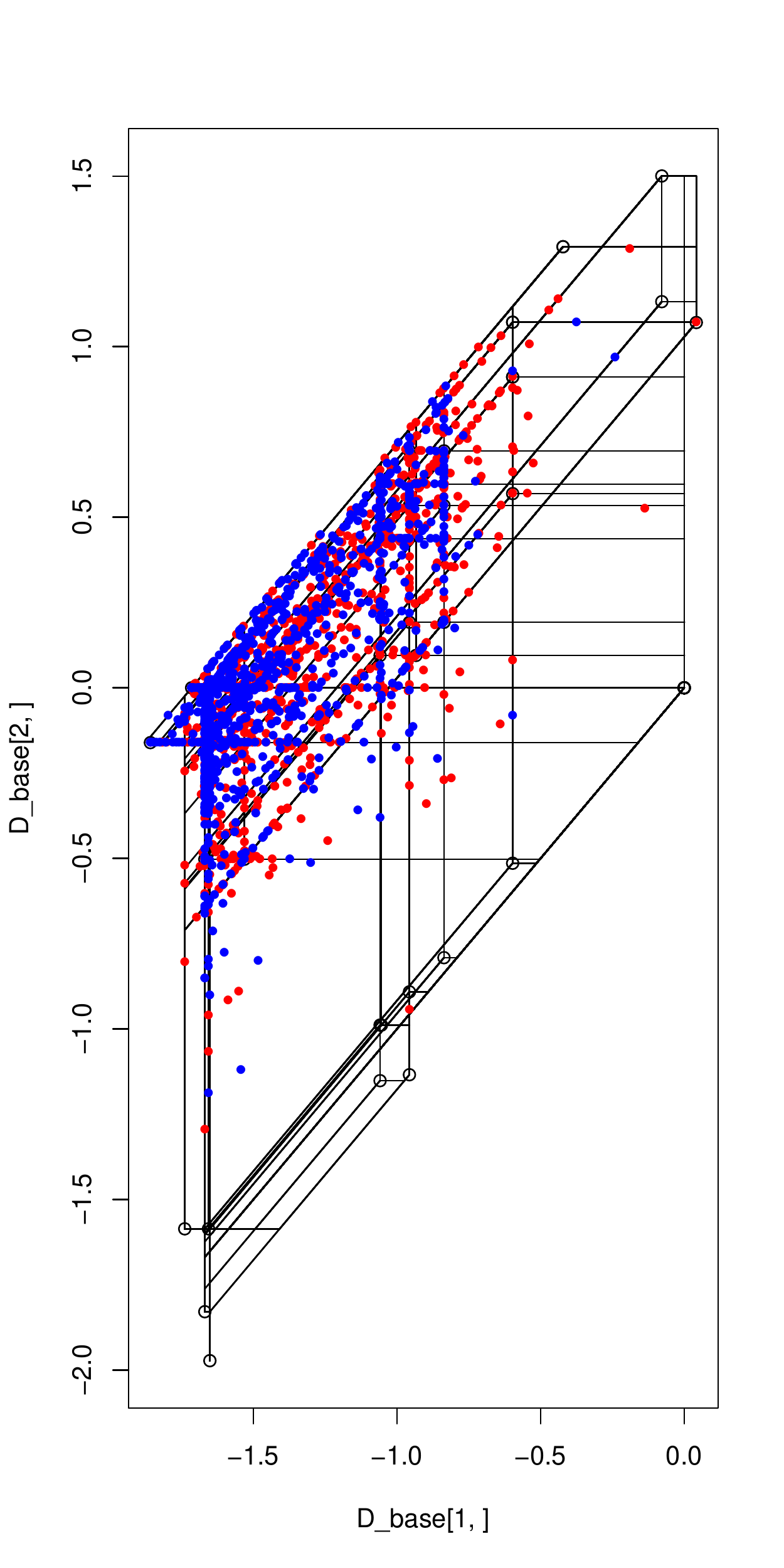}
        \caption{$r = 0.25$.}
        \label{fig:TroS2025}
    \end{subfigure}
    ~ %add desired spacing between images, e. g. ~, \quad, \qquad, \hfill etc. 
    \begin{subfigure}[b]{0.3\textwidth}
        \includegraphics[width=0.9\textwidth]{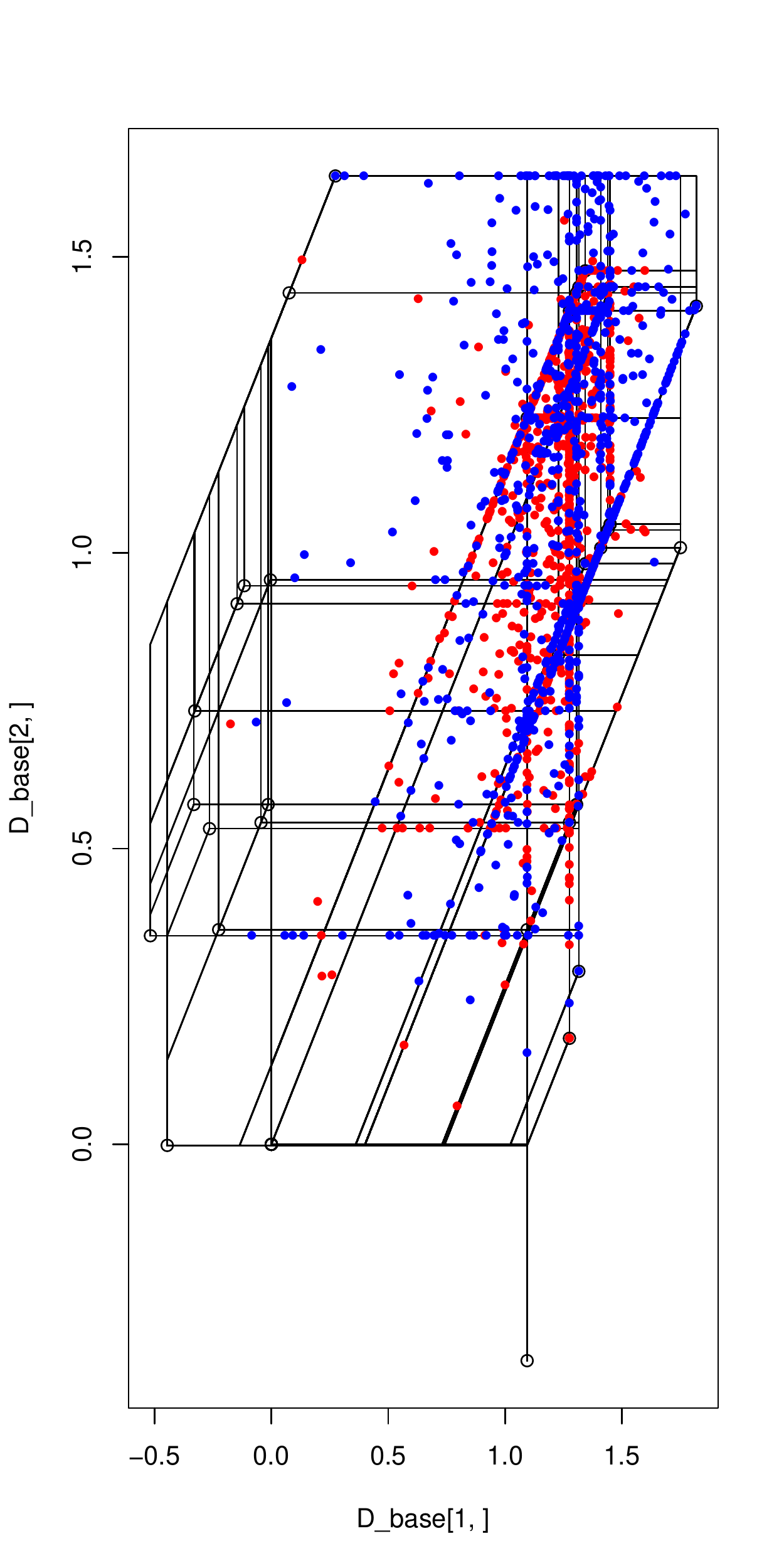}
        \caption{$r = 0.5$.}
        \label{fig:TroS205}
    \end{subfigure}
    ~ %add desired spacing between images, e. g. ~, \quad, \qquad, \hfill etc. 
    \begin{subfigure}[b]{0.3\textwidth}
        \includegraphics[width=0.9\textwidth]{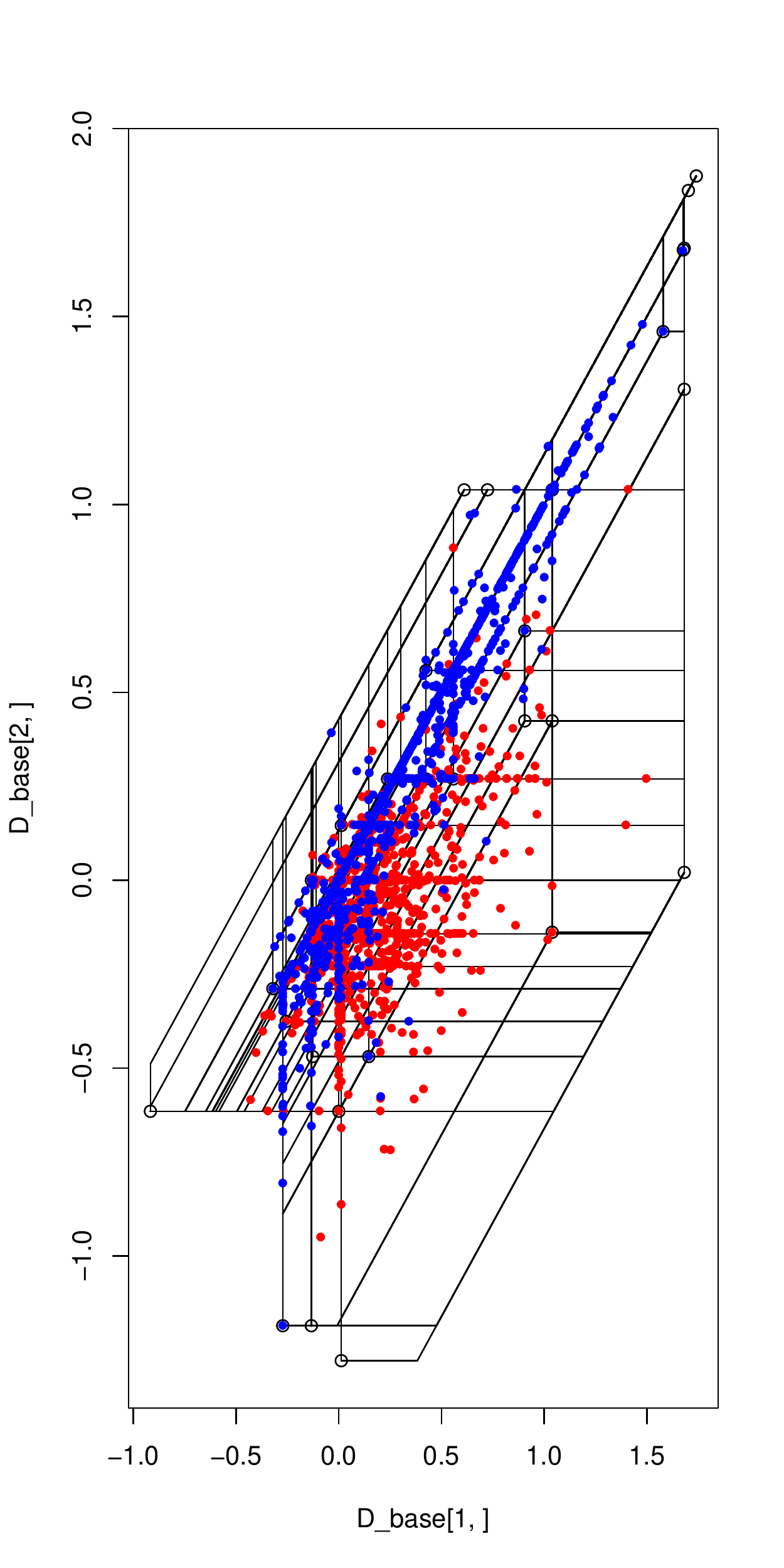}
        \caption{$r = 1$.}
        \label{fig:TroS21}
    \end{subfigure}
 ~
    \begin{subfigure}[b]{0.3\textwidth}
        \includegraphics[width=0.9\textwidth]{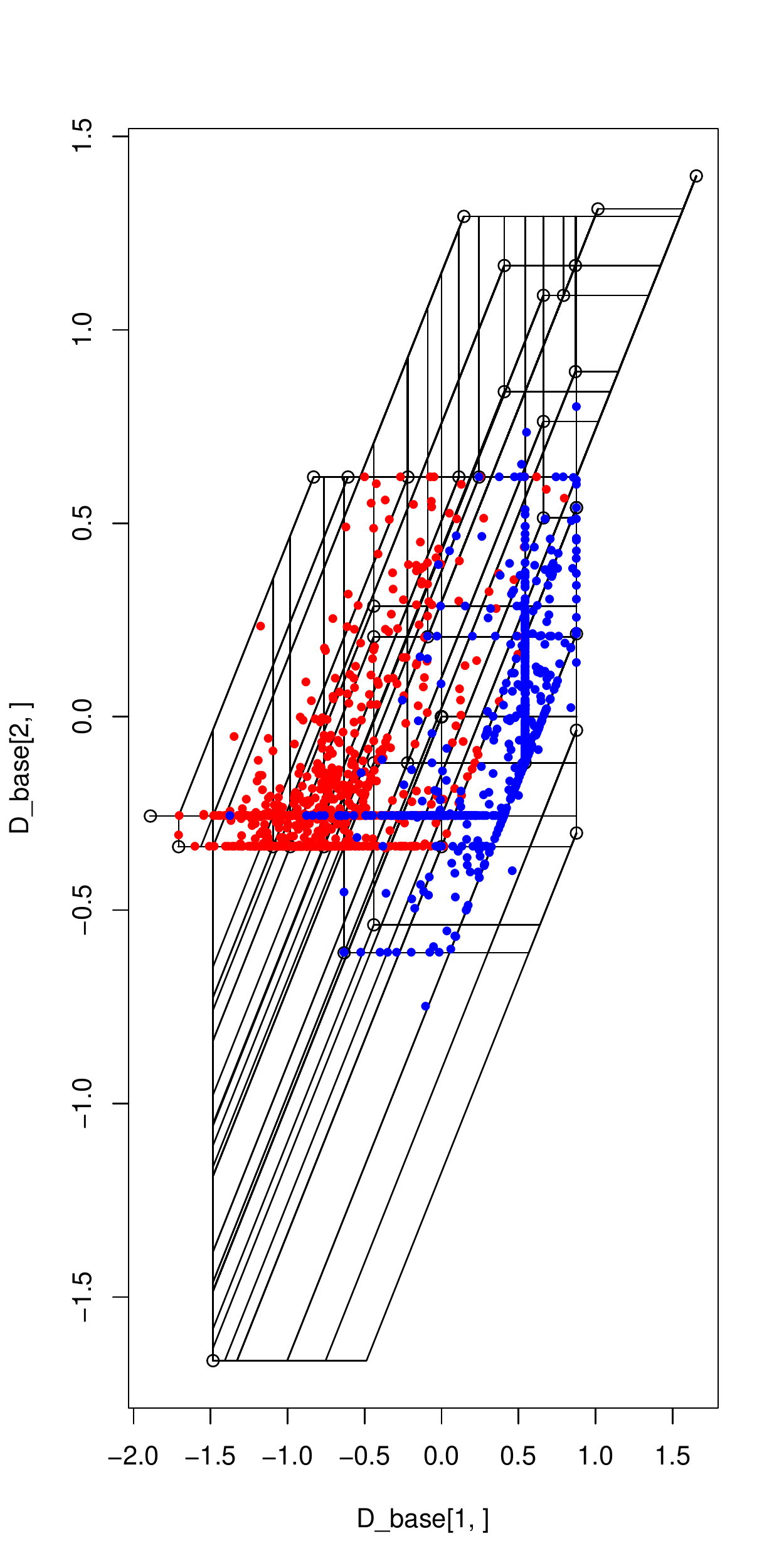}
        \caption{$r = 2$.}
        \label{fig:TroS22}
    \end{subfigure}
    ~ %add desired spacing between images, e. g. ~, \quad, \qquad, \hfill etc. 
    \begin{subfigure}[b]{0.3\textwidth}
        \includegraphics[width=0.9\textwidth]{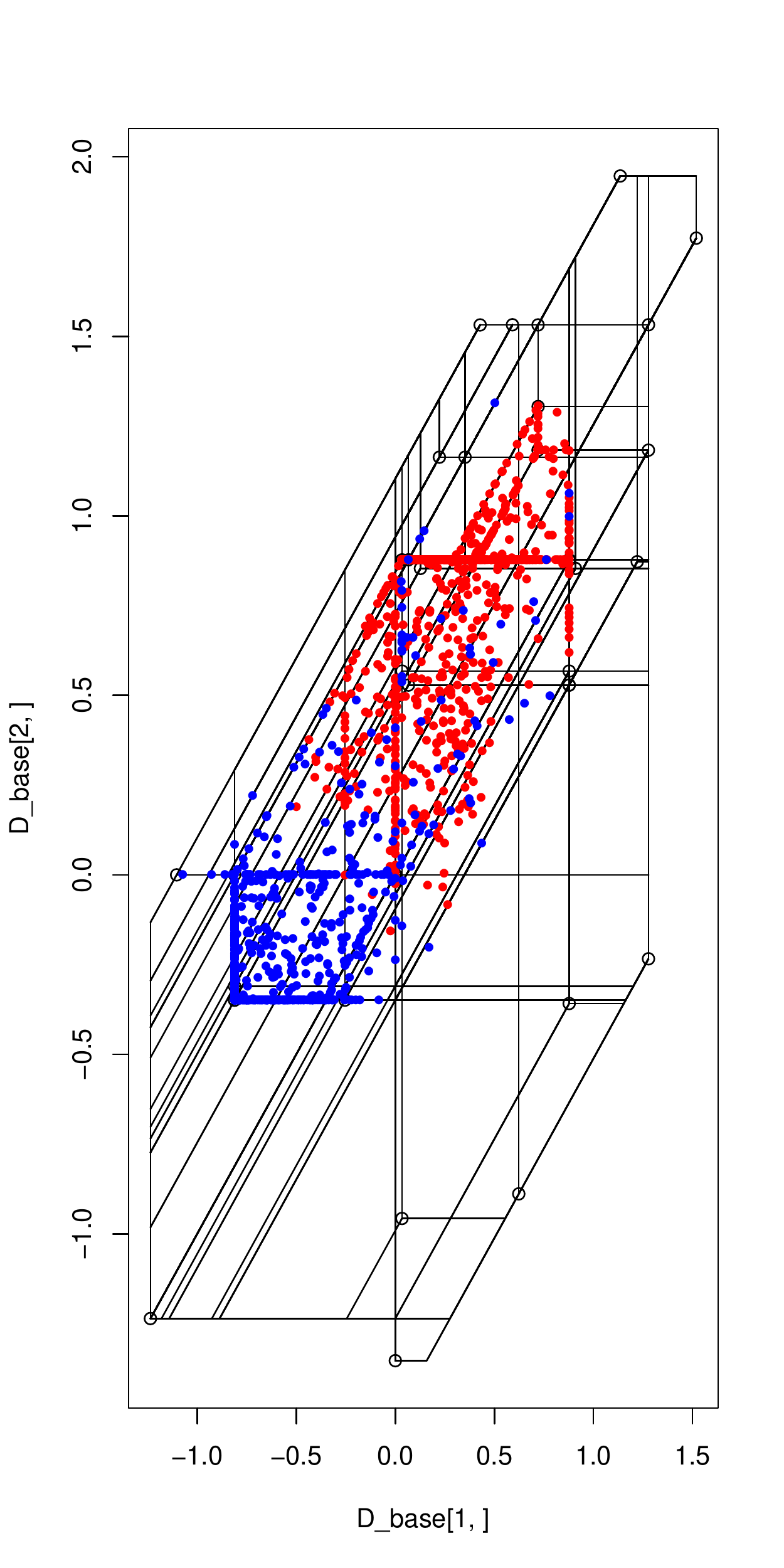}
        \caption{$r = 5$.}
        \label{fig:TroS25}
    \end{subfigure}
    ~ %add desired spacing between images, e. g. ~, \quad, \qquad, \hfill etc. 
    \begin{subfigure}[b]{0.3\textwidth}
        \includegraphics[width=0.9\textwidth]{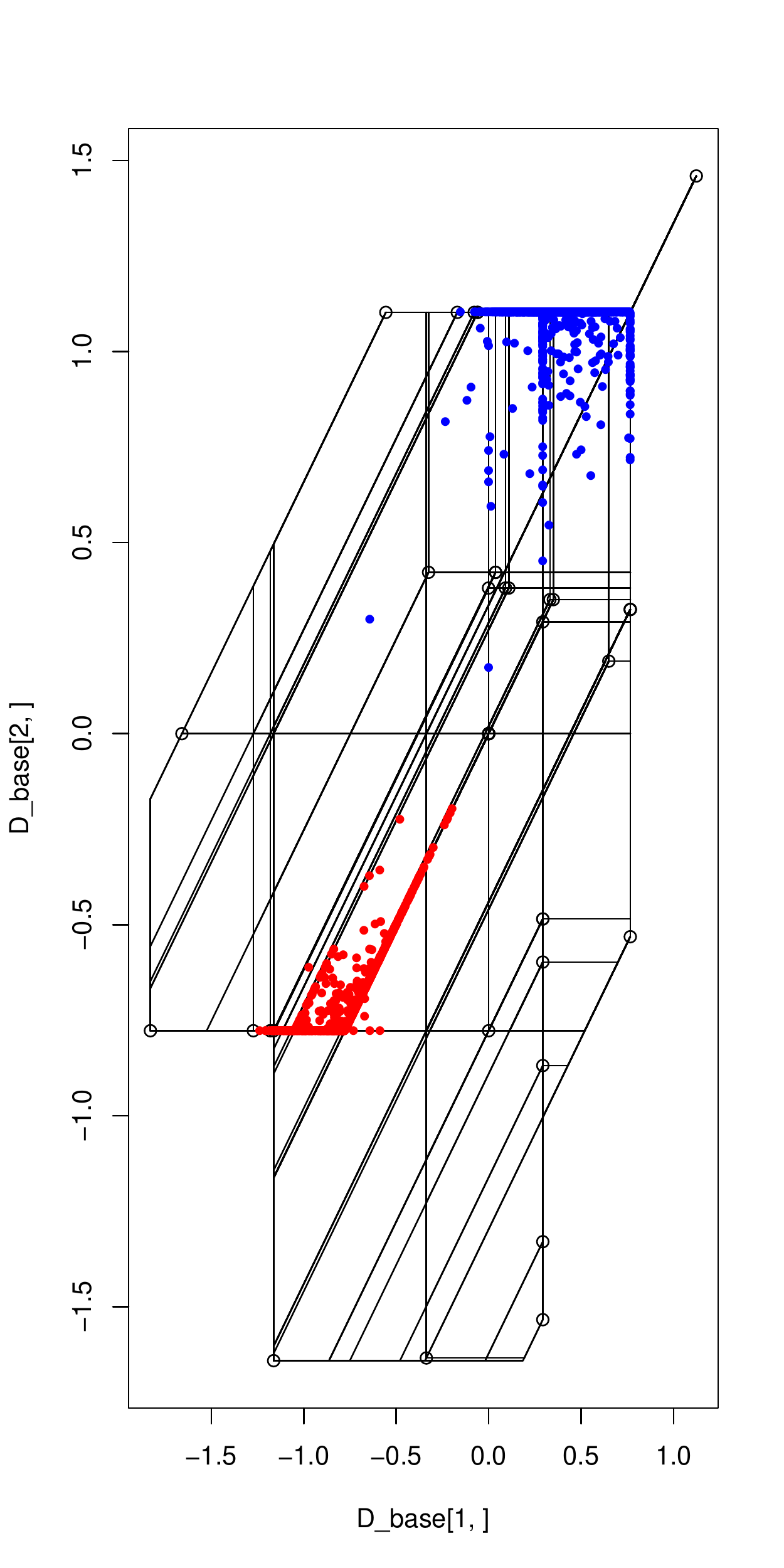}
        \caption{$r = 10$.}
        \label{fig:TroS210}
    \end{subfigure}
    \caption{We applied tropical PCAs on the mixture of two coalescent
      distributions using Algorithm \ref{al1}.  We colored blue for
      projected trees whose gene trees 
          are generated from one coalescent distribution and red for
          the other distribution. We varied the ratio
          $r$.}\label{fig:troPCA}
\end{figure}

\begin{figure}
    \centering
    \begin{subfigure}[b]{0.3\textwidth}
        \includegraphics[width=0.9\textwidth]{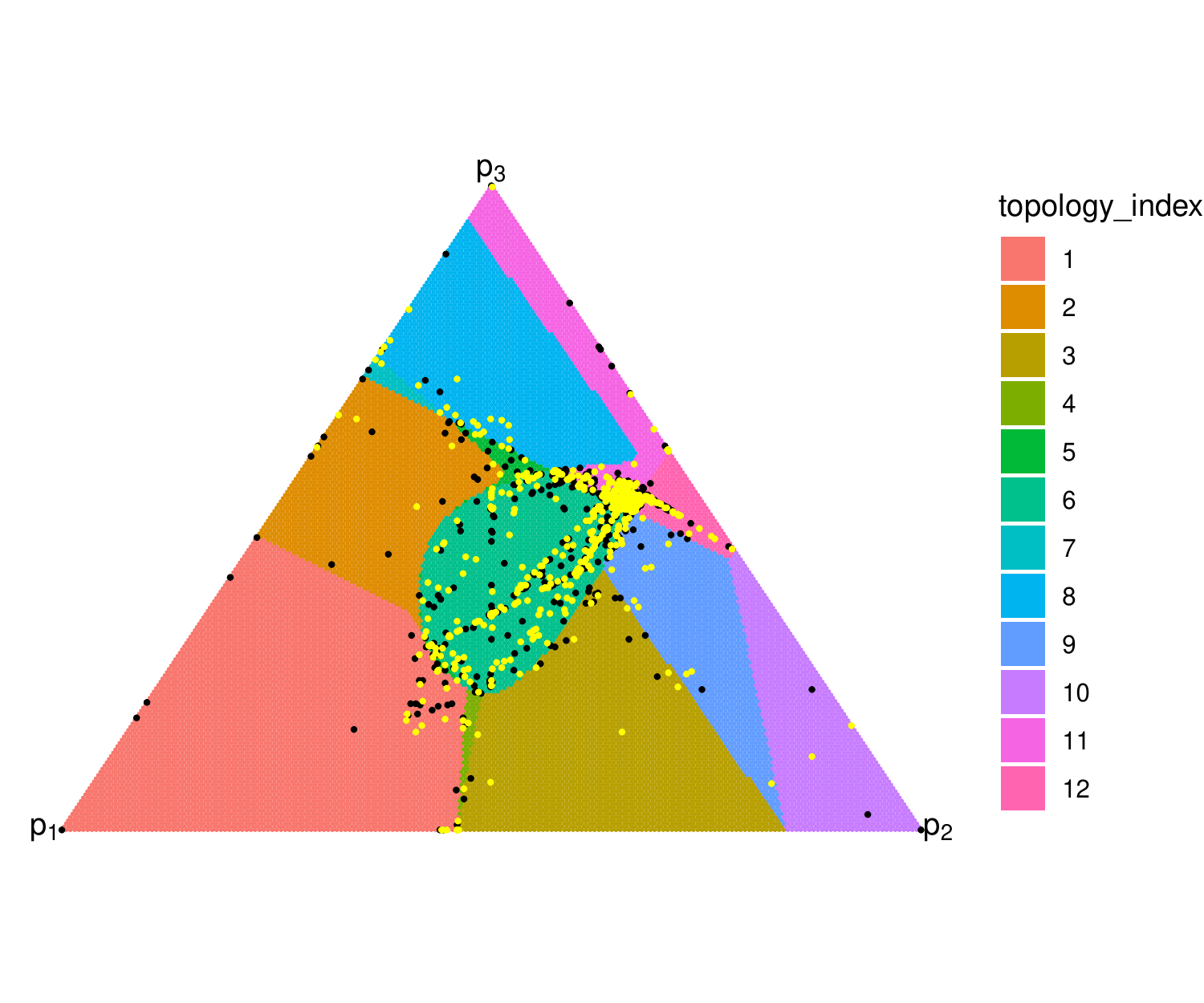}
        \caption{$r = 0.25$.}
        \label{fig:BHVS2025}
    \end{subfigure}
    ~ %add desired spacing between images, e. g. ~, \quad, \qquad, \hfill etc. 
    \begin{subfigure}[b]{0.3\textwidth}
        \includegraphics[width=0.9\textwidth]{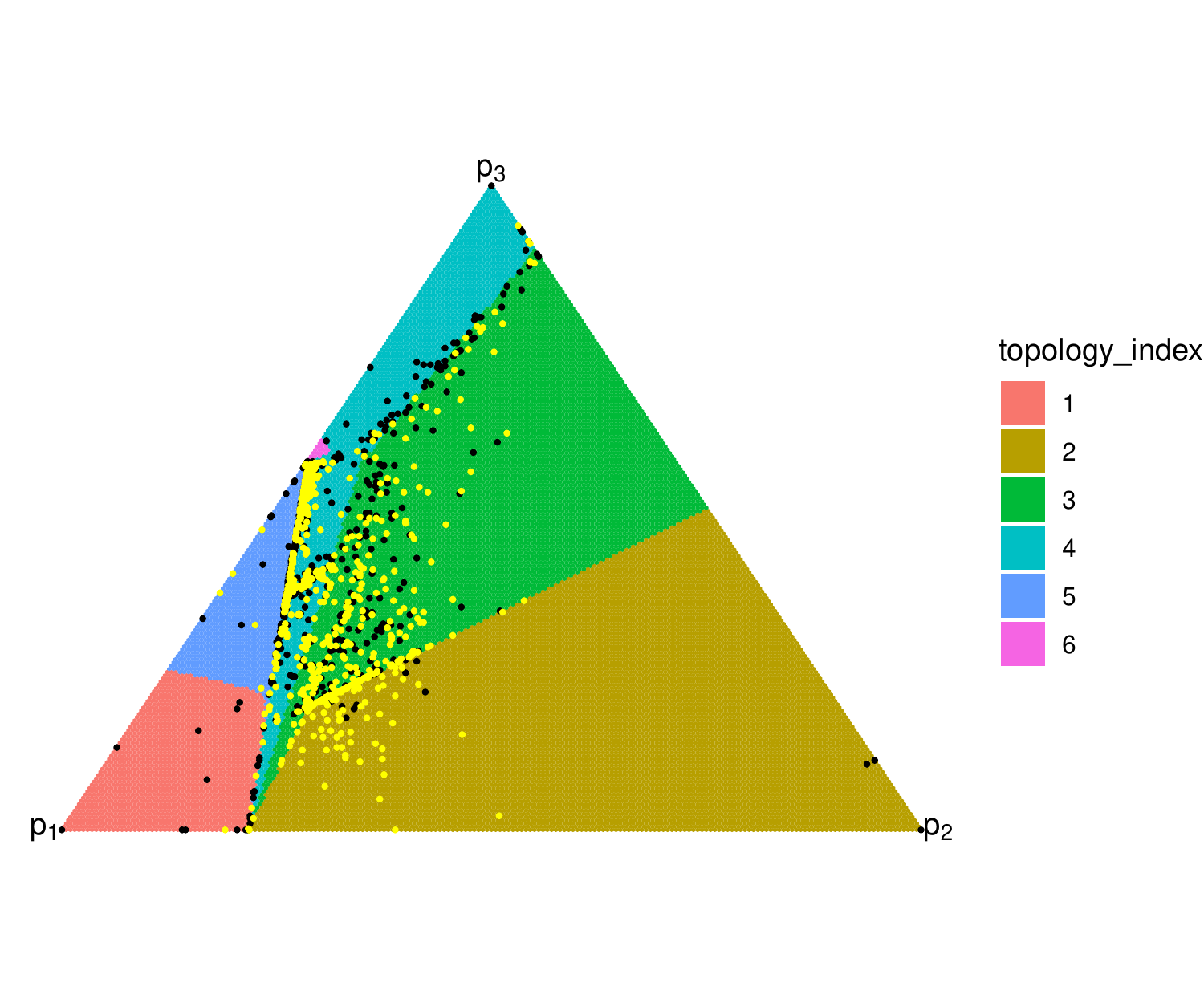}
        \caption{$r = 0.5$.}
        \label{fig:BHVS205}
    \end{subfigure}
    ~ %add desired spacing between images, e. g. ~, \quad, \qquad, \hfill etc. 
    \begin{subfigure}[b]{0.3\textwidth}
        \includegraphics[width=0.9\textwidth]{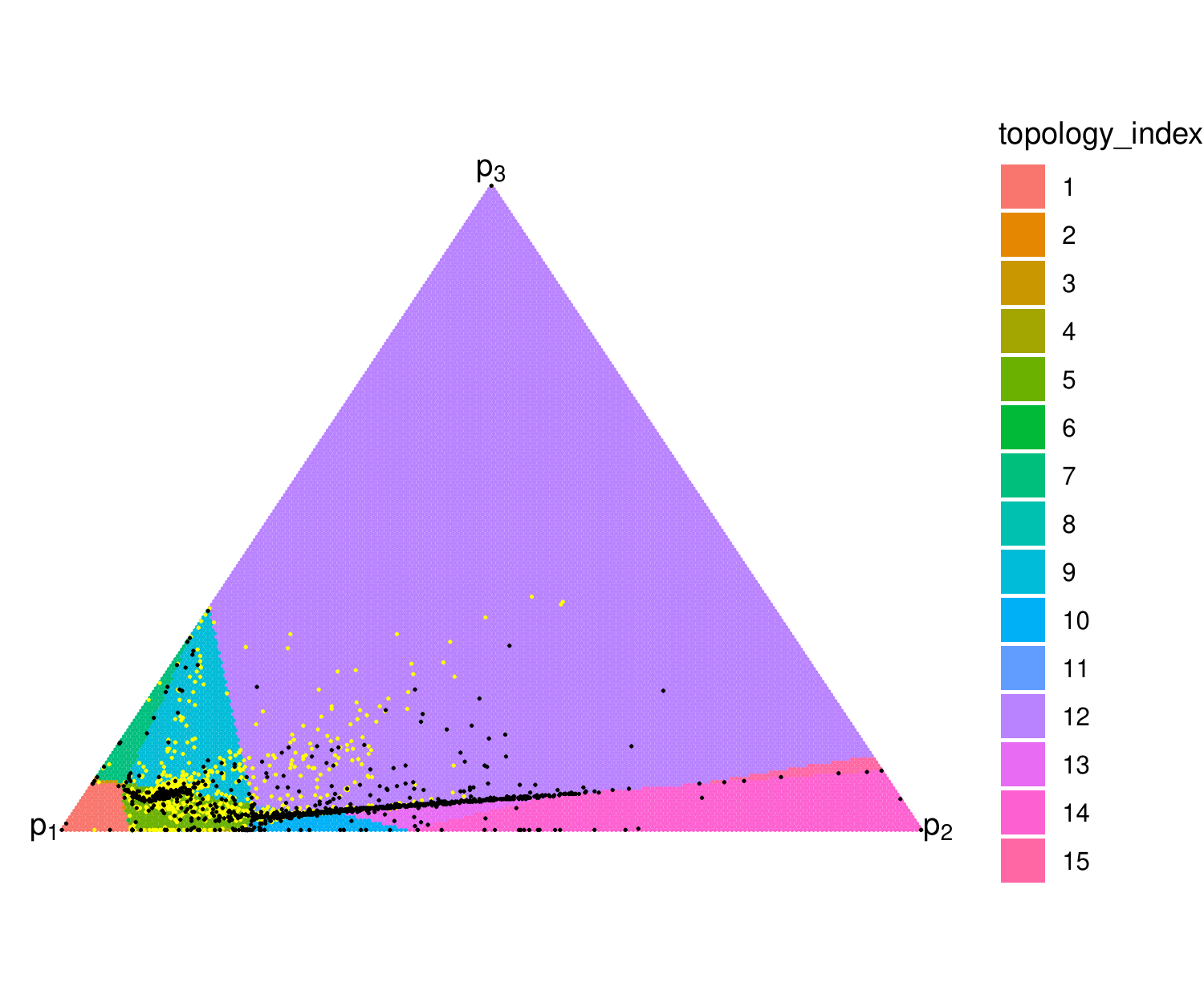}
        \caption{$r = 1$.}
        \label{fig:BHVS21}
    \end{subfigure}
 ~
    \begin{subfigure}[b]{0.3\textwidth}
        \includegraphics[width=0.9\textwidth]{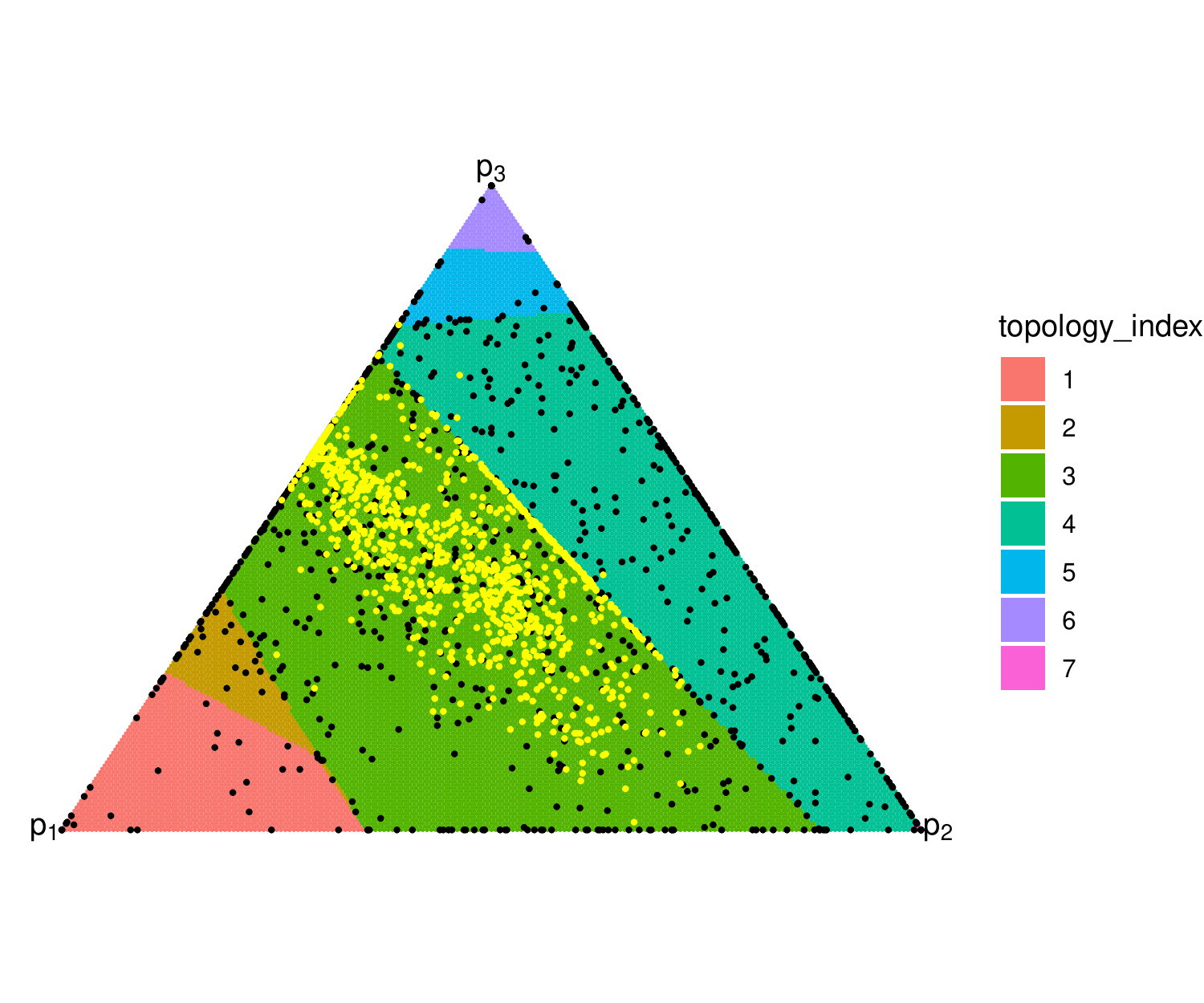}
        \caption{$r = 2$.}
        \label{fig:BHVS22}
    \end{subfigure}
    ~ %add desired spacing between images, e. g. ~, \quad, \qquad, \hfill etc. 
    \begin{subfigure}[b]{0.3\textwidth}
        \includegraphics[width=0.9\textwidth]{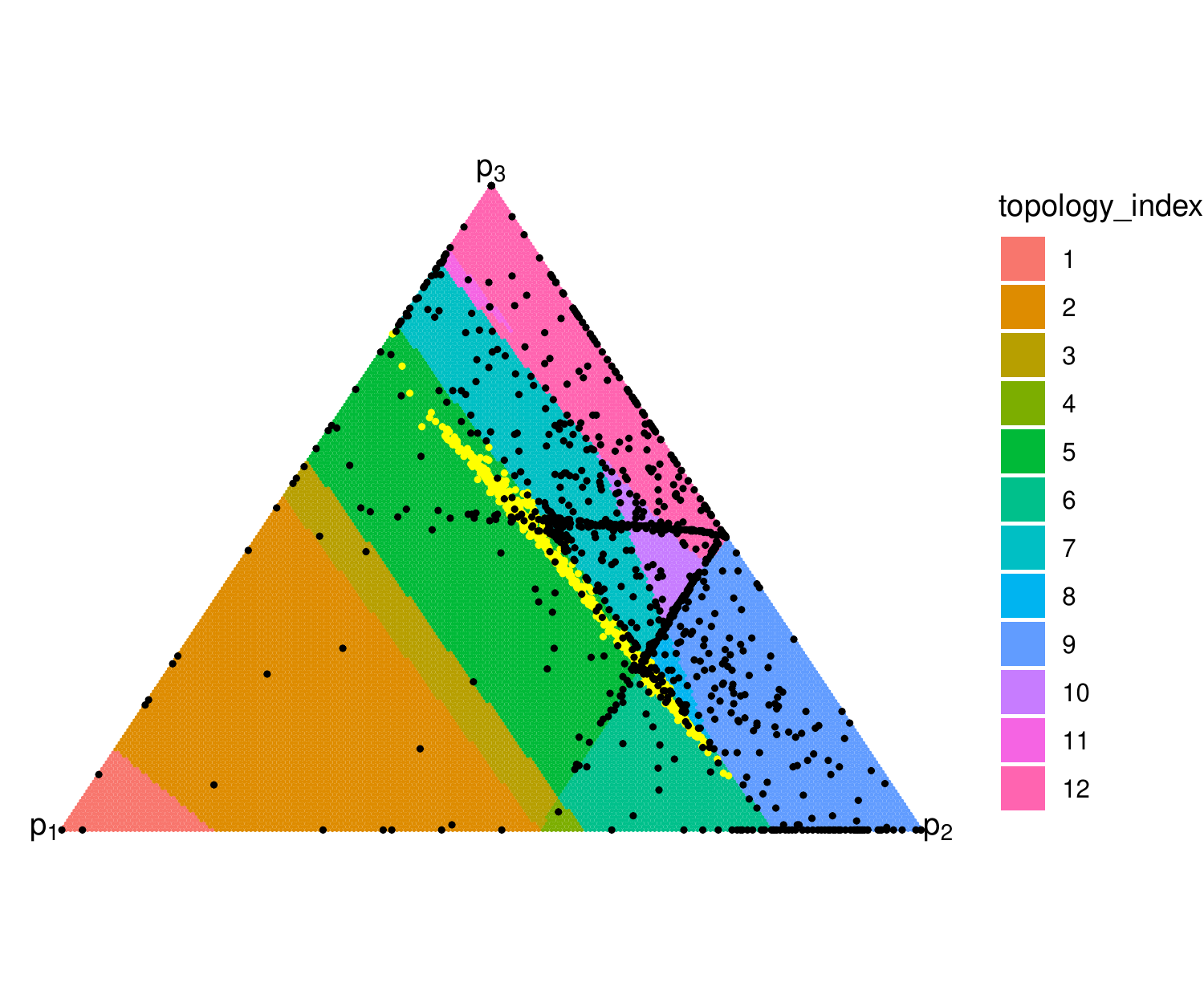}
        \caption{$r = 5$.}
        \label{fig:BHVS25}
    \end{subfigure}
    ~ %add desired spacing between images, e. g. ~, \quad, \qquad, \hfill etc. 
    \begin{subfigure}[b]{0.3\textwidth}
        \includegraphics[width=0.9\textwidth]{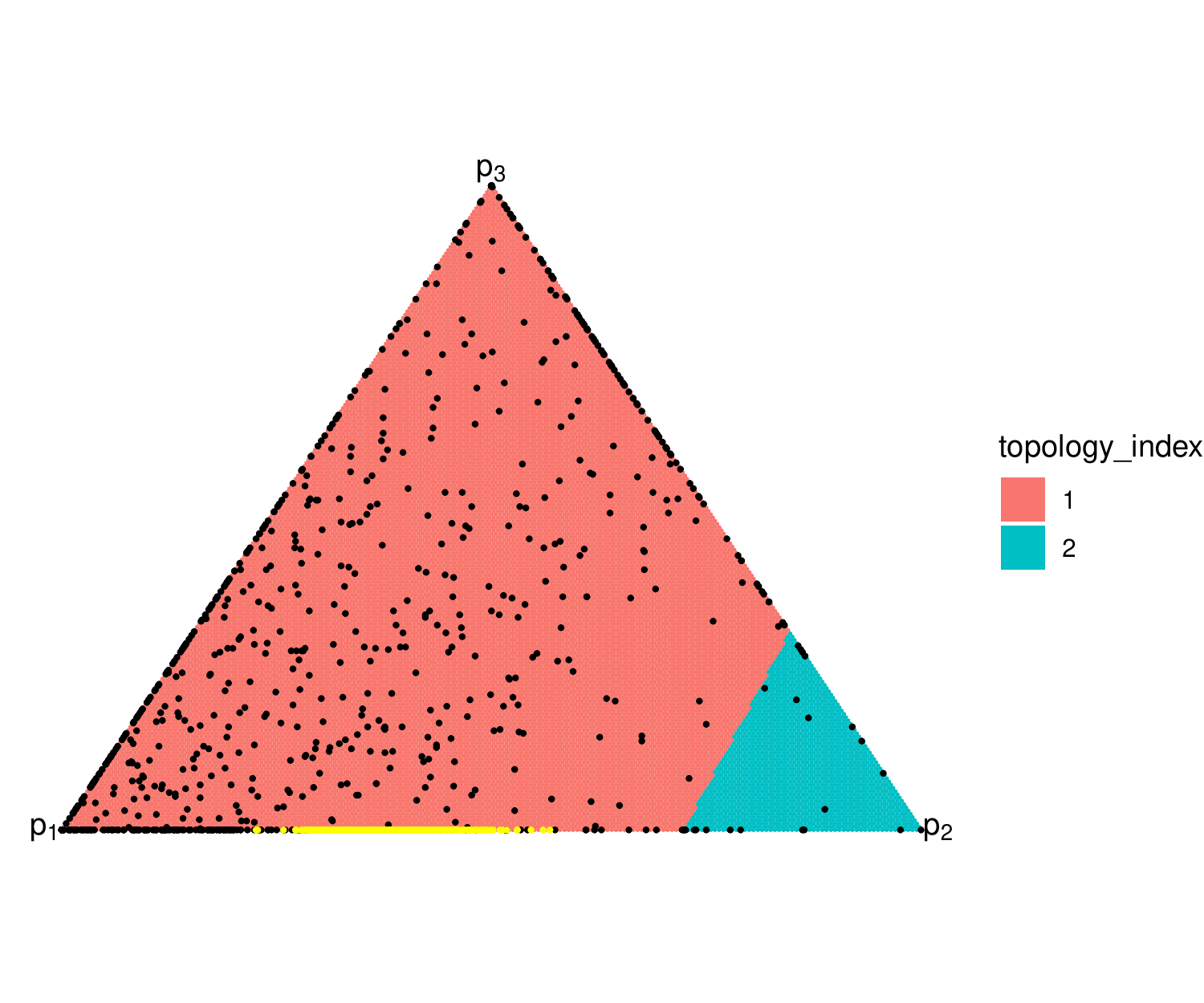}
        \caption{$r = 10$.}
        \label{fig:BHVS210}
    \end{subfigure}
    \caption{We applied BHV PCAs on the mixture of two coalescent
      distributions.  We colored black for
      projected trees whose gene trees 
          are generated from one coalescent distribution and yellow for
          the other distribution. We varied the ratio
          $r$.}\label{fig:BHVPCA}
\end{figure}

\begin{table}[h!]
\begin{center}
\begin{tabular}{|c|c|c|}
\hline
$r$ & Tropical PCA  & BHV PCA  \\ \hline
 $0.25$ & $0.316$ &  $0.009$ \\ \hline
 $0.5$ & $0.297$ &  $0.009$    \\ \hline
 $1$& $0.186$ &   $0.042$  \\ \hline
 $2$& $0.319$ &   $0.034$  \\ \hline
 $5$&  $0.278$ &   $0.009$  \\ \hline
 $10$& $0.396$ &   $0.009$  \\ \hline
\end{tabular}
\caption{$R^2$ for the tropical PCAs and the BHV PCAs for the mixtures
of two coalescent distributions. }
\end{center}
\end{table}

\subsection{Sensitivity analysis}

For the second simulation study, we worked on sensitivity analysis of
our MCMC method for estimating tropical PCA from the given data.  Also
with simulated data we study a convergence rate of our MCMC approach. 

For a sensitivity analysis, we run our MCMC
approach 10 times on the same set of data (fixed the number of
leaves, the number of trees, and the number of iterations for each MCMC) and see how the vertices of the tropical PCA
and corresponding $R^2$ change.  %We compare vertices of a tropical triangle for
%each run with the maximum and minimum of Robinson-Foulds distances as
%well as tropical distances for each pair of vertices from different runs.  
For a convergence rate, we plot $R^2$ varying the number of iterations
on each MCMC run.  

We conducted these simulations on a set of random tree topologies and a
set of datapoints of fixed tree topology.  The parameters for this simulation
study are listed in Table \ref{tab1}.  For each case, we ran ten Markov
chains.  

\begin{table}[h!]
\begin{center}
\begin{tabular}{|c|c|}\hline
 Parameter&   \\ \hline
 Tree topology& random or fixed tree topology  \\\hline
 Number of leaves& $4, \ldots , 9$ \\\hline
 Number of trees & 5, 25, 50, 100, 1000\\\hline
 Number of iterations & 10, 100, 1000, 10000\\\hline
\end{tabular}
\caption{Parameters set up for the second simulation
  study.}\label{tab1}
\end{center}
\end{table}

\subsubsection{Fixed tree topology}

We fixed the equidistant  tree topology shown in Figure \ref{fig:fixedtree}.  The
biggest external edge from the leaf 1 to its root has its branch
length 1.  
\begin{figure}
    \centering
        \includegraphics[width=0.3\textwidth]{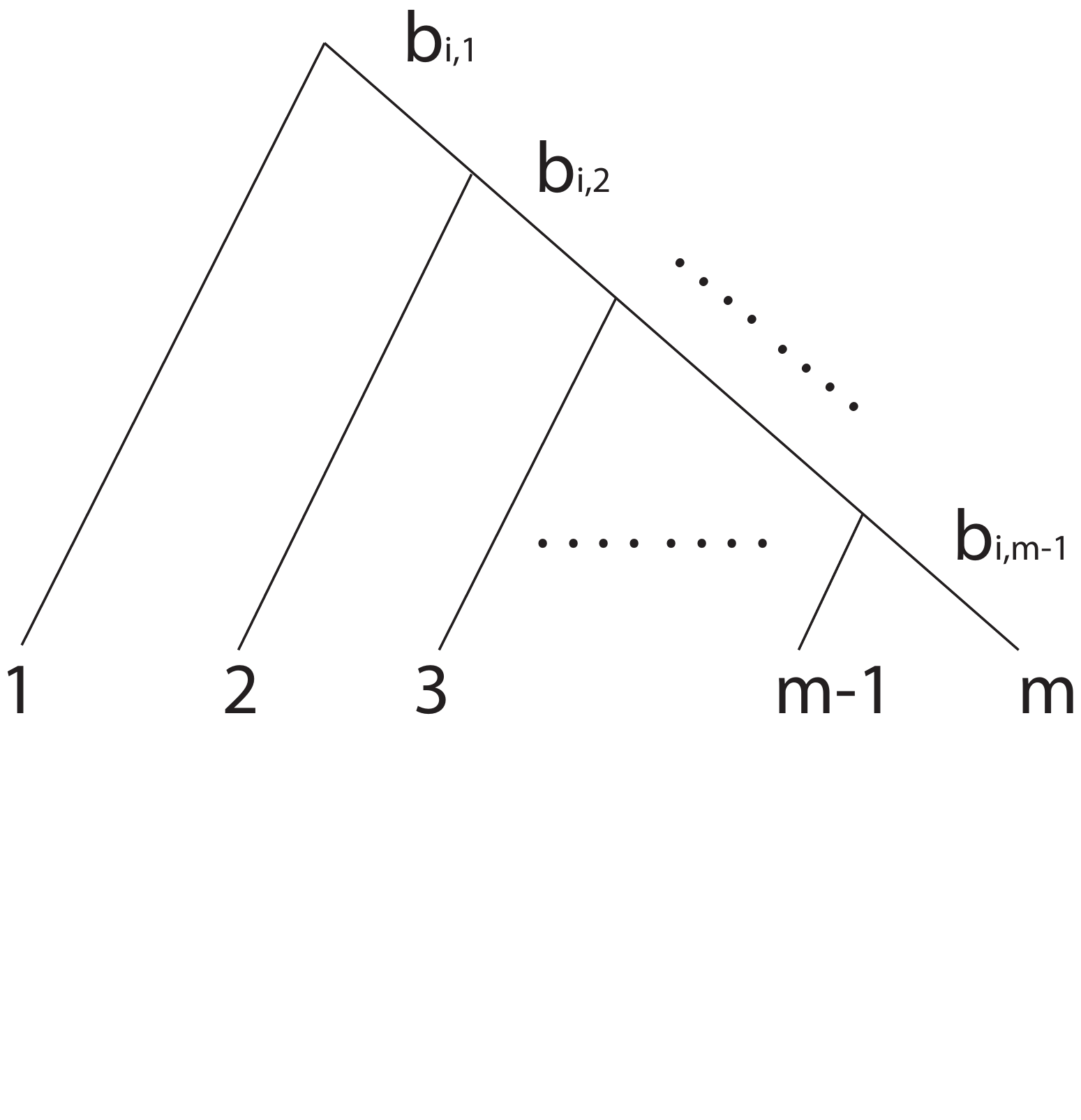}
%\vskip -1in
    \caption{The fixed equidistant  tree topology for the second simulation study.}\label{fig:fixedtree}
\end{figure}

We generate a set of random equidistant trees with this fixed tree topology shown
in Figure  \ref{fig:fixedtree} by Algorithm \ref{randtree}.

\begin{algorithm}[Generating random equidistant trees with the fixed tree topology
  shown in \ref{fig:fixedtree}]\label{randtree}  \qquad
\begin{itemize}
\item Input: The number of leaves $m$ and sample size $n$.
\item Output: $n$ many random trees with the fixed tree topology with
  $m$ leaves.
\end{itemize}
\begin{algorithmic}
\For{$i = 1, \ldots , n$}
       \State Generate an equidistant tree $T_i$ with its topology shown in
       \ref{fig:fixedtree}.
       \State Let $b_{i, 1}, \ldots , b_{i, m-1}$ be interior edges in
       $T_i$ shown in Figure \ref{fig:fixedtree}.  Let $l(b_{i, j})$
       is the branch length of the edge $b_{i, j}$. 
       \For{$j = 1, \ldots , m-1$}
               \State Pick a random number $u$ uniformly $[0,
               1-(\sum_{k=1}^{j-1}l(b_{i, k}))]$.
               \State Assign $l(b_{i, j}) = u$.
         \EndFor
       \For{$j = 1, \ldots , m-1$}
               \State Assign the branch length for the external edge
               adjacent to a leaf $j$ equal to $\sum_{k =
                 j}^{m-1}l(b_{i, j})$.
         \EndFor
\EndFor
\State Return $\{T_1, \ldots , T_n\}$.
\end{algorithmic}
\end{algorithm}

The simulation results can be found in Figure \ref{fig:sens1} below.  Note
that even though we fix the tree topology the branch lengths are
random.  Therefore, in general $R^2$ should be low.  In addition, from
the figure note that the more iterations for a Markov chain correlates with a higher $R^2$.   Also from
the figure the variation of $R^2$ is smaller when the
number of iterations increases.  This
implies that a MCMC is converging to a local or global optima.  Since the
variance is very small for some cases, a Markov chain seems to converge to a
global optima for these cases.  From this
experiments, a Markov chain is quickly converging to a local or global
optima very quickly.  Also in general when the number of leaves is smaller
we have higher $R^2$ because the dimension of the tree space is
smaller; similarly when the number of trees is smaller we also have a higher $R^2$.
\begin{figure}[!h]
    \centering
    \begin{subfigure}[b]{0.3\textwidth}
        \includegraphics[width=0.9\textwidth]{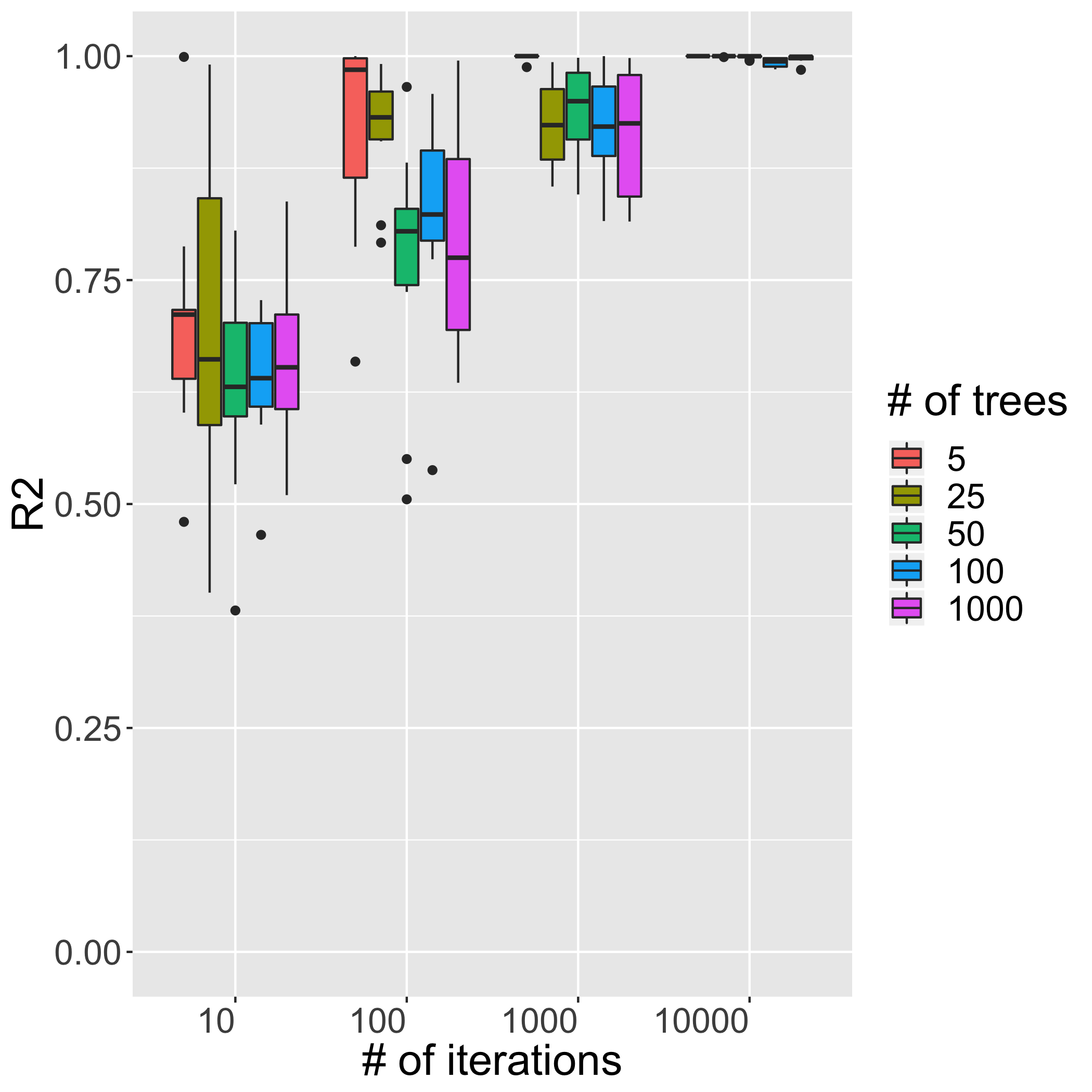}
        \caption{Leaves ($m$) = 4.}
        \label{fig:same4}
    \end{subfigure}
    ~ %add desired spacing between images, e. g. ~, \quad, \qquad, \hfill etc. 
    \begin{subfigure}[b]{0.3\textwidth}
        \includegraphics[width=0.9\textwidth]{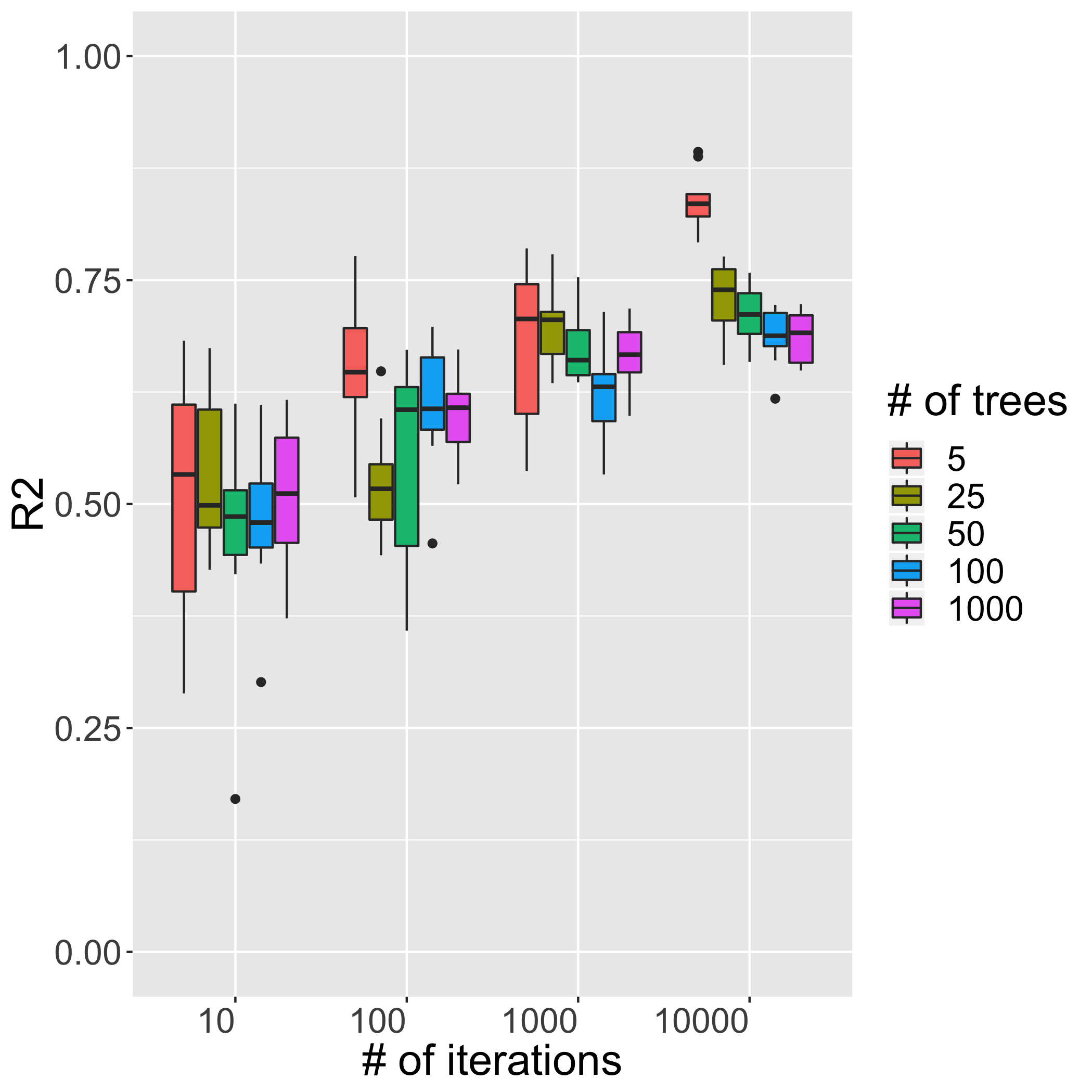}
        \caption{Leaves ($m$) = 5.}
        \label{fig:same5}
    \end{subfigure}
    ~ %add desired spacing between images, e. g. ~, \quad, \qquad, \hfill etc. 
    \begin{subfigure}[b]{0.3\textwidth}
        \includegraphics[width=0.9\textwidth]{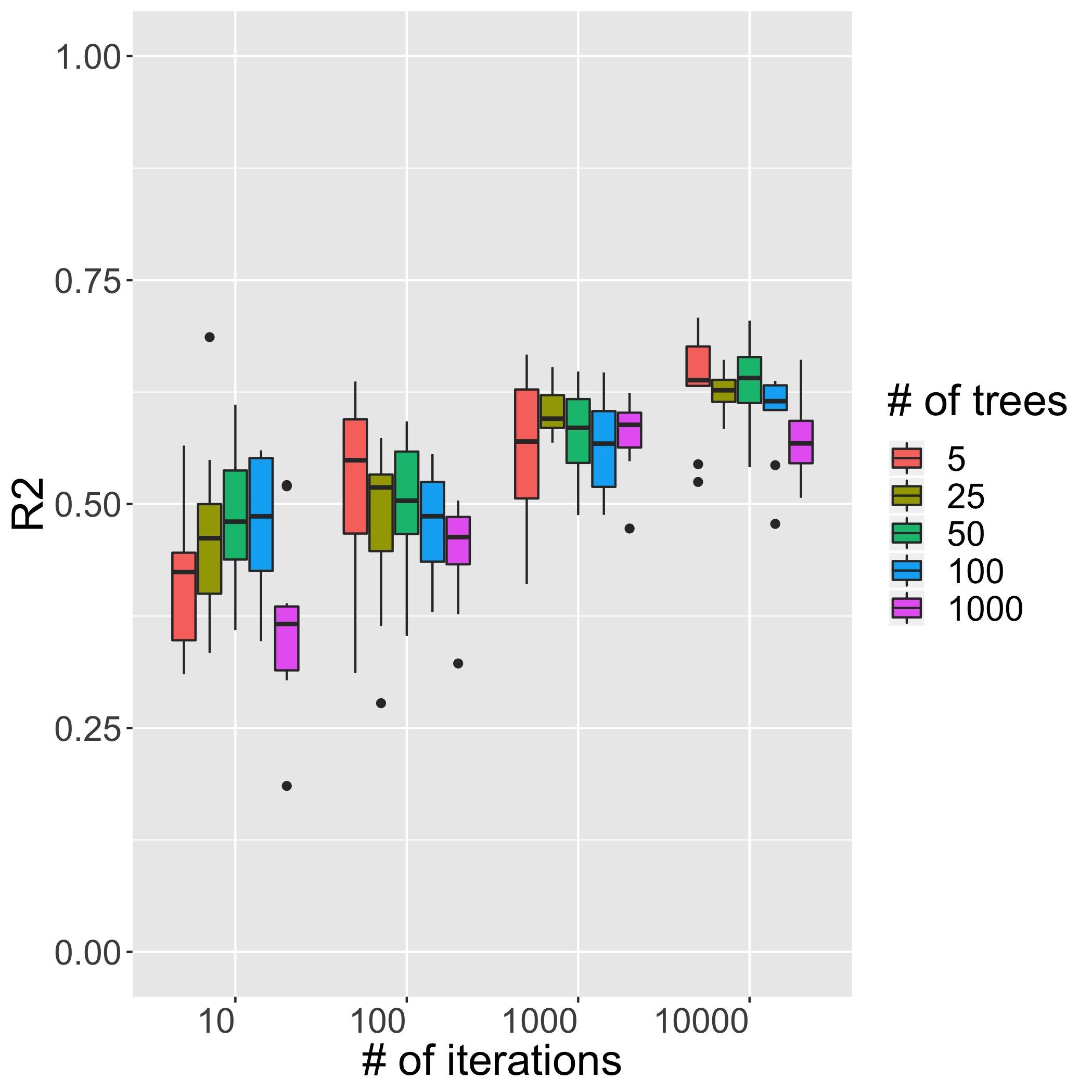}
        \caption{Leaves ($m$) = 6.}
        \label{fig:same6}
    \end{subfigure}
 ~
    \begin{subfigure}[b]{0.3\textwidth}
        \includegraphics[width=0.9\textwidth]{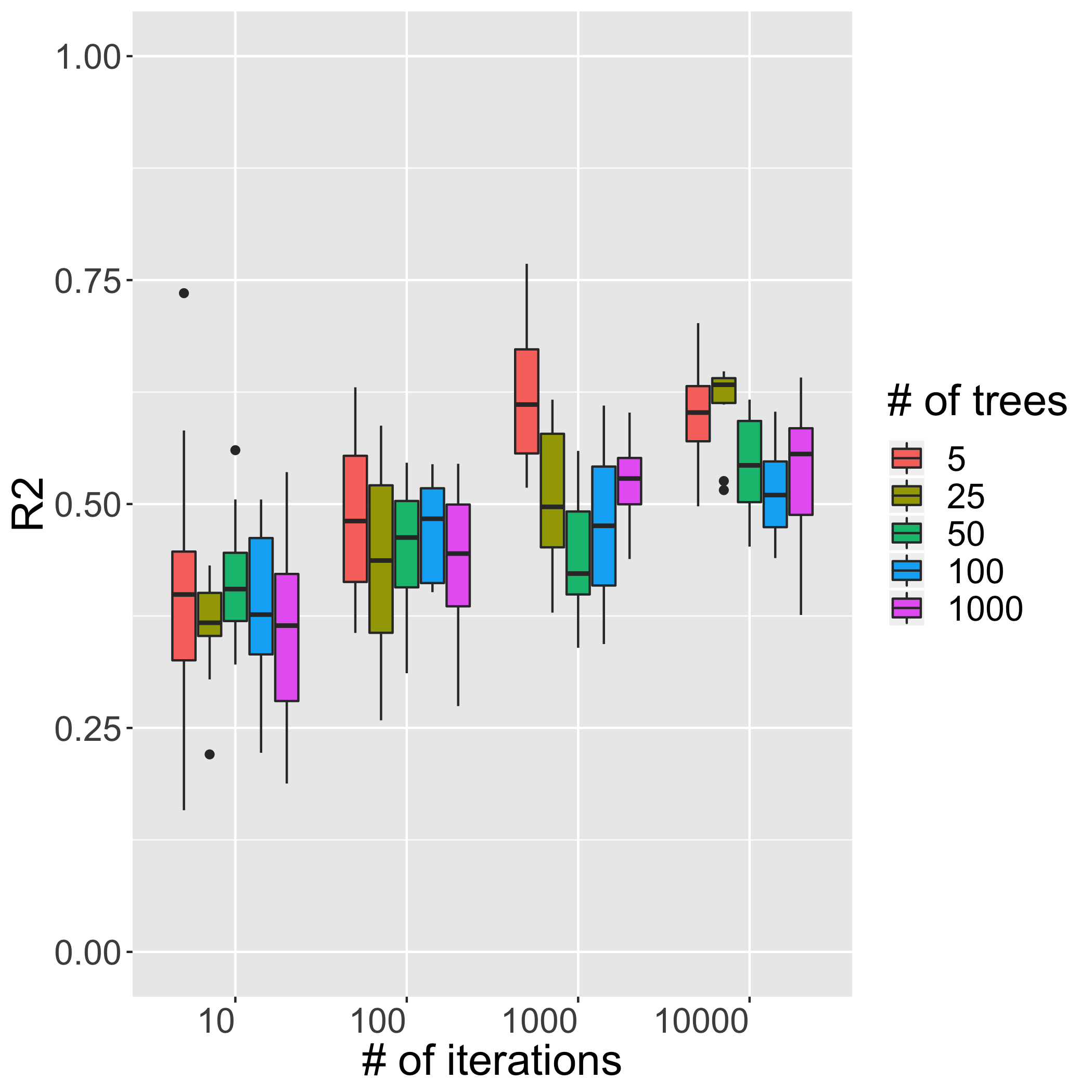}
        \caption{Leaves ($m$) = 7.}
        \label{fig:same7}
    \end{subfigure}
    ~ %add desired spacing between images, e. g. ~, \quad, \qquad, \hfill etc. 
    \begin{subfigure}[b]{0.3\textwidth}
        \includegraphics[width=0.9\textwidth]{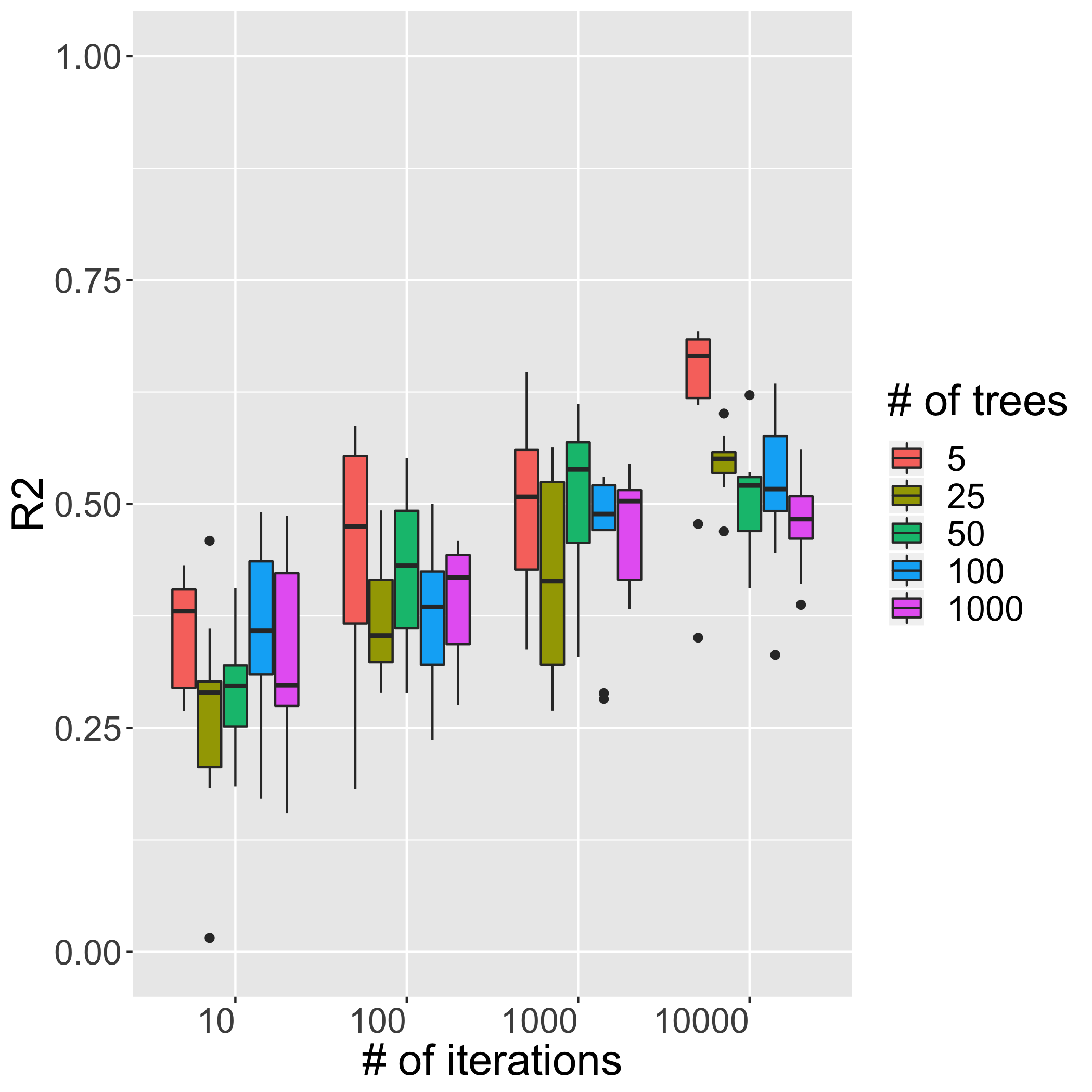}
        \caption{Leaves ($m$) = 8.}
        \label{fig:same8}
    \end{subfigure}
    ~ %add desired spacing between images, e. g. ~, \quad, \qquad, \hfill etc. 
    \begin{subfigure}[b]{0.3\textwidth}
        \includegraphics[width=0.9\textwidth]{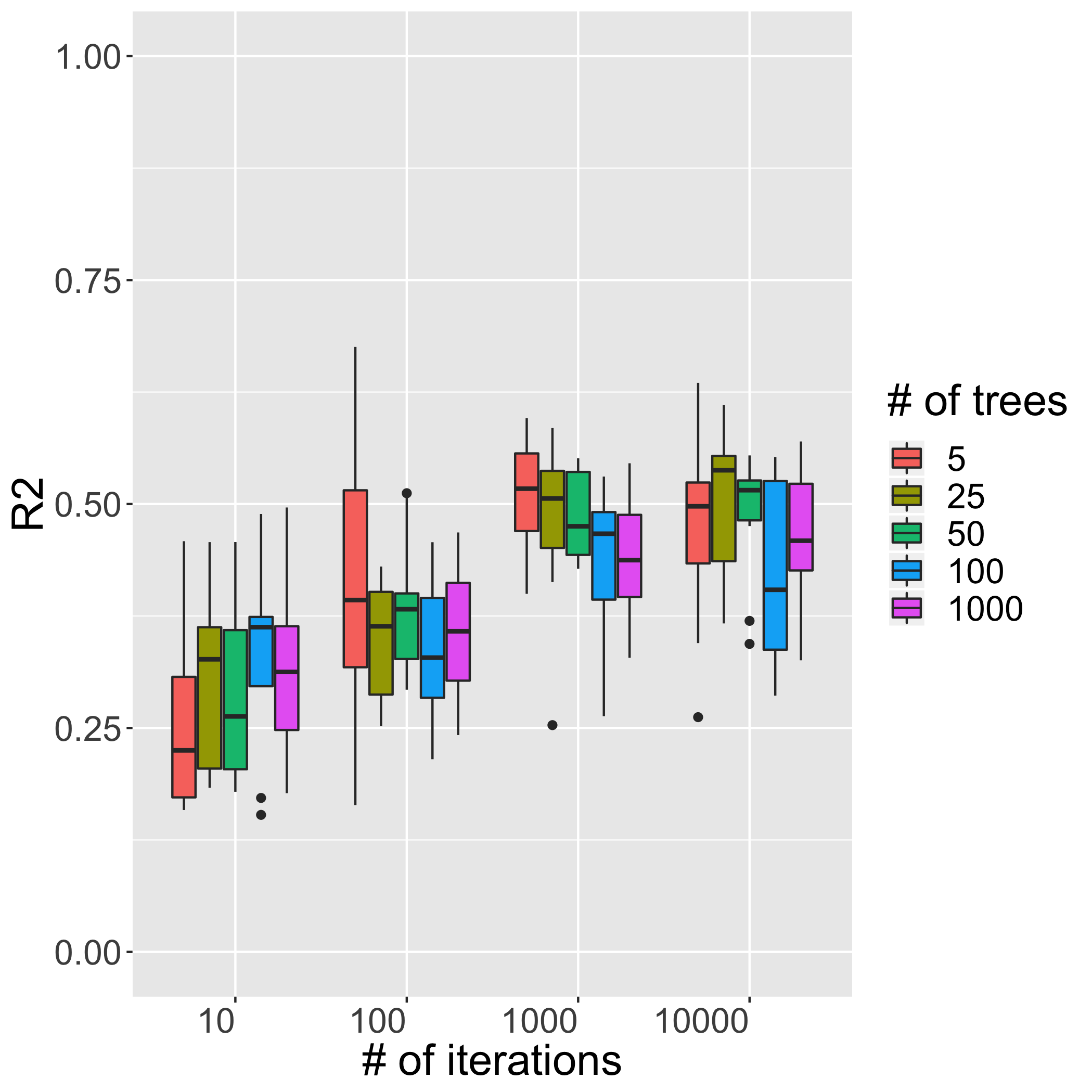}
        \caption{Leaves ($m$) = 9.}
        \label{fig:same9}
    \end{subfigure}
    \caption{We applied the MCMC approach to compute a tropical PCA on
      the datasets generated by Algorithm \ref{randtree} with the
      parameters listed in Table \ref{tab1}. We ran 10 Markov chains
      and we computed a box plot for each case. The y-axis represents
      $R^2$ and the x-axis represents the number of iterations for a MCMC.}\label{fig:sens1}
\end{figure}

\subsubsection{Random tree topology}

For this simulation we used {\tt rcoal()} function from {\tt ape}
package in R with fixed height equal to one.  
The simulation results can be found in Figure \ref{fig:sens2} below.  Note
that since trees are all random, in general $R^2$ should be much lower
compared to $R^2$ computed from the samples of ultrametrics with the
same tree topologies.
The overall trends of $R^2$ are similar to the ones with the fixed tree
topology.  
\begin{figure}[!h]
    \centering
    \begin{subfigure}[b]{0.3\textwidth}
        \includegraphics[width=0.9\textwidth]{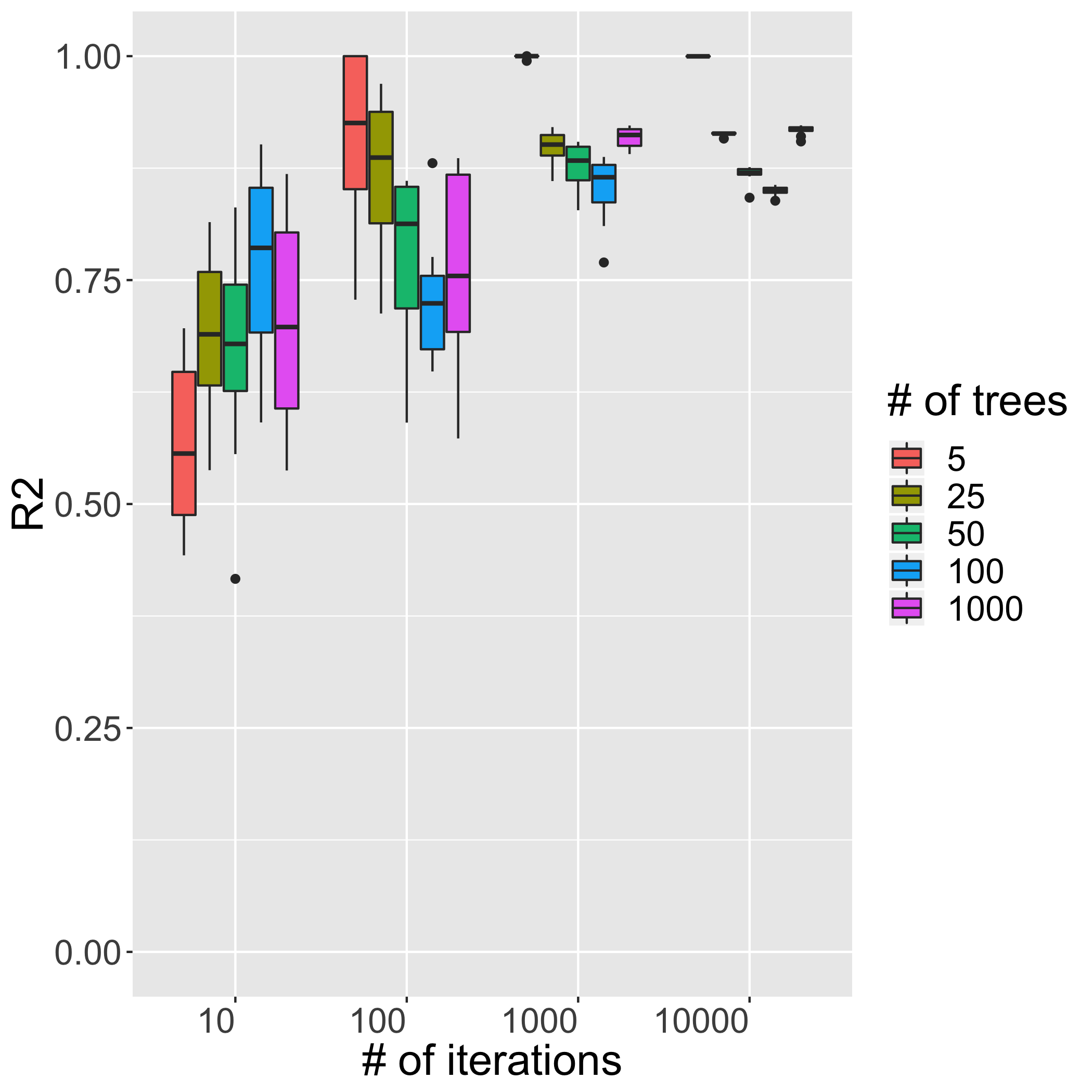}
        \caption{Leaves ($m$) = 4.}
        \label{fig:rand4}
    \end{subfigure}
    ~ %add desired spacing between images, e. g. ~, \quad, \qquad, \hfill etc. 
    \begin{subfigure}[b]{0.3\textwidth}
        \includegraphics[width=0.9\textwidth]{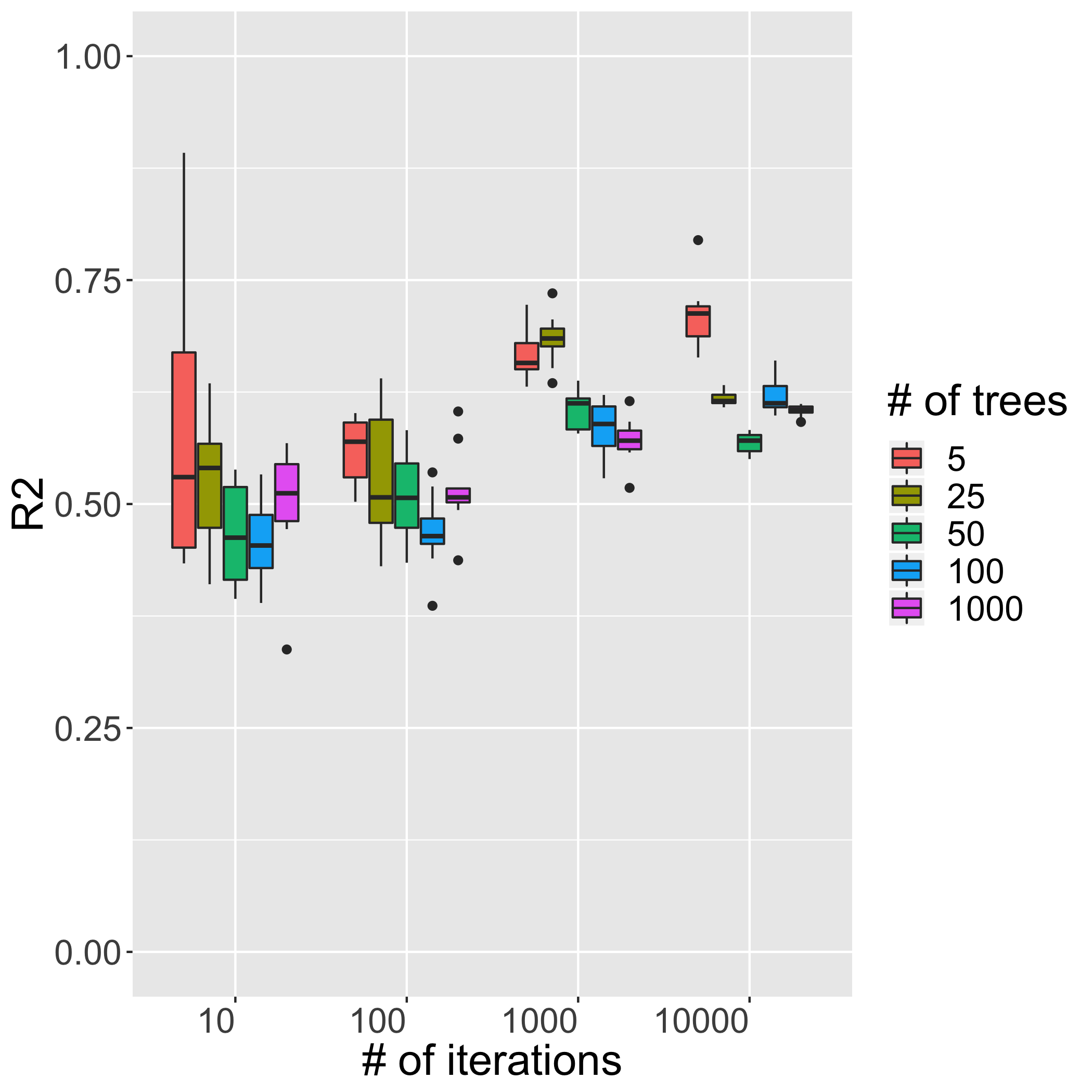}
        \caption{Leaves ($m$) = 5.}
        \label{fig:rand5}
    \end{subfigure}
    ~ %add desired spacing between images, e. g. ~, \quad, \qquad, \hfill etc. 
    \begin{subfigure}[b]{0.3\textwidth}
        \includegraphics[width=0.9\textwidth]{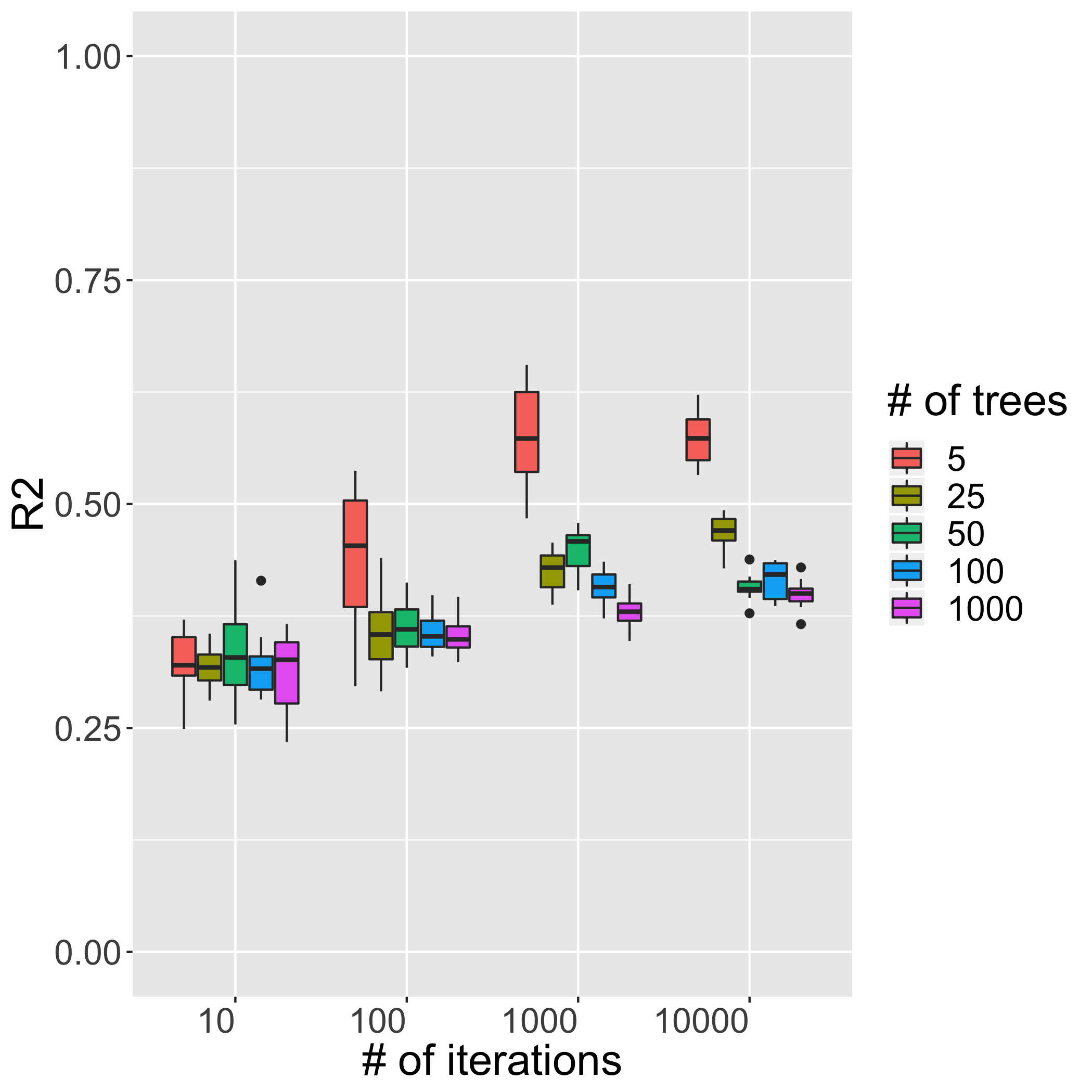}
        \caption{Leaves ($m$) = 6.}
        \label{fig:rand6}
    \end{subfigure}
 ~
    \begin{subfigure}[b]{0.3\textwidth}
        \includegraphics[width=0.9\textwidth]{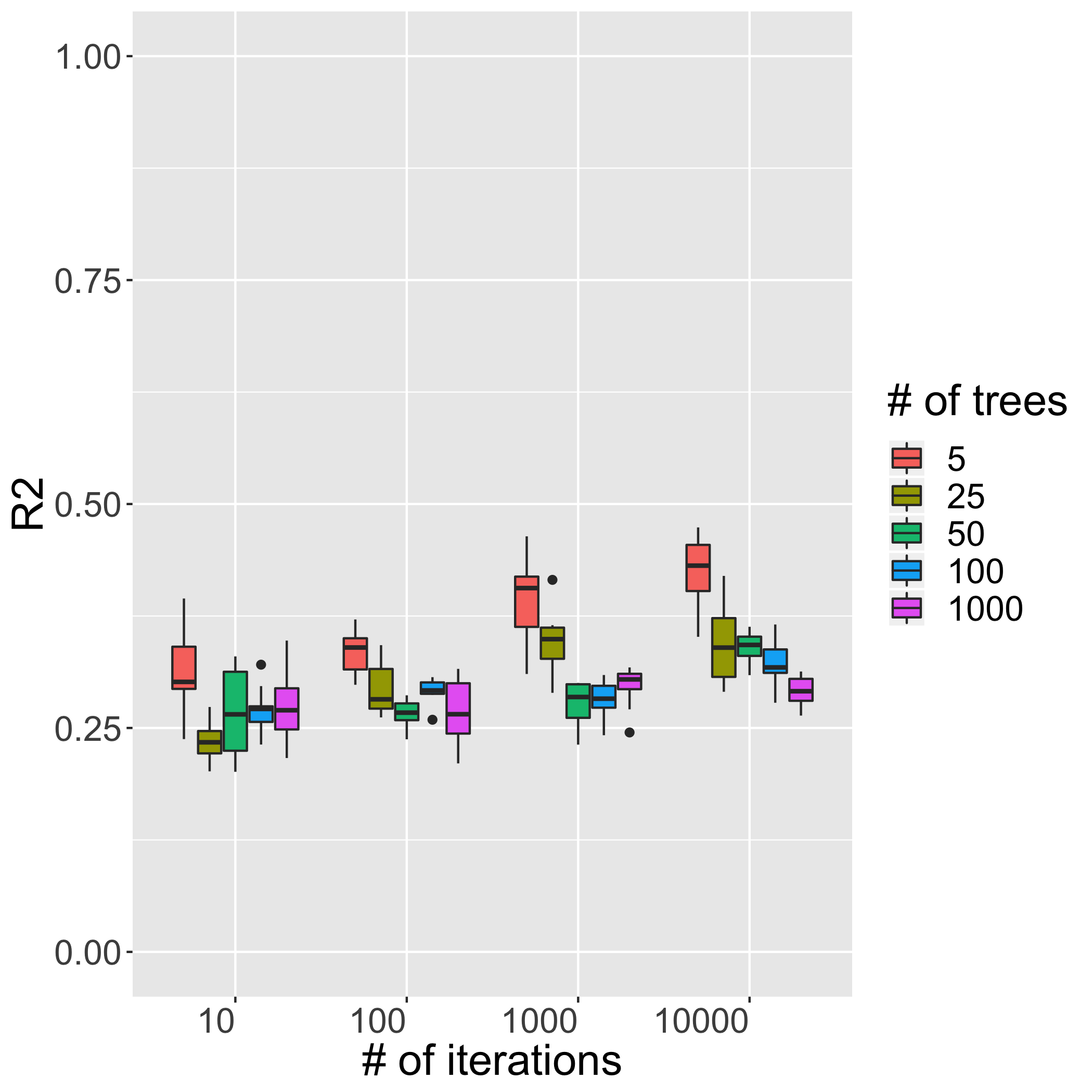}
        \caption{Leaves ($m$) = 7.}
        \label{fig:rand7}
    \end{subfigure}
    ~ %add desired spacing between images, e. g. ~, \quad, \qquad, \hfill etc. 
    \begin{subfigure}[b]{0.3\textwidth}
        \includegraphics[width=0.9\textwidth]{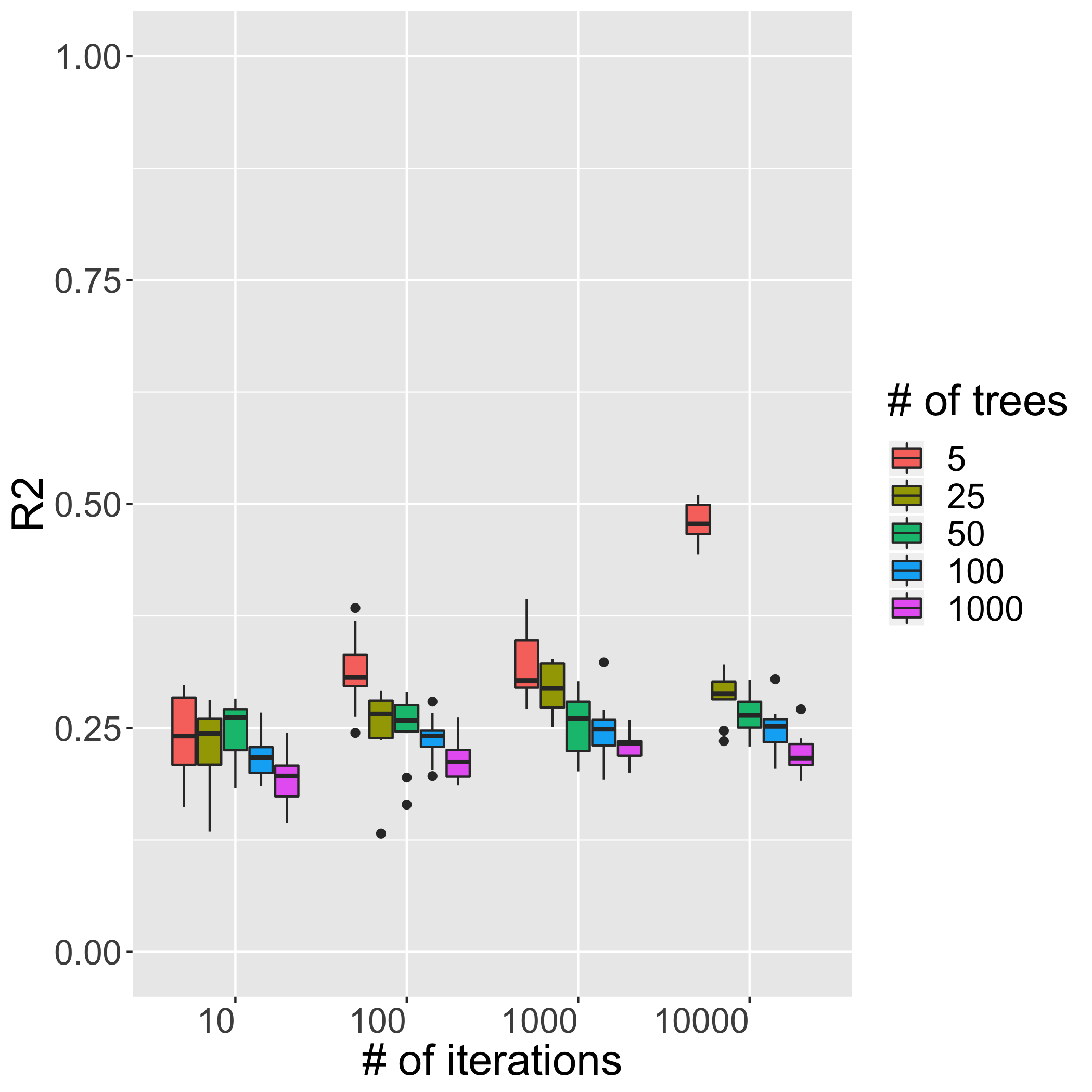}
        \caption{Leaves ($m$) = 8.}
        \label{fig:rand8}
    \end{subfigure}
    ~ %add desired spacing between images, e. g. ~, \quad, \qquad, \hfill etc. 
    \begin{subfigure}[b]{0.3\textwidth}
        \includegraphics[width=0.9\textwidth]{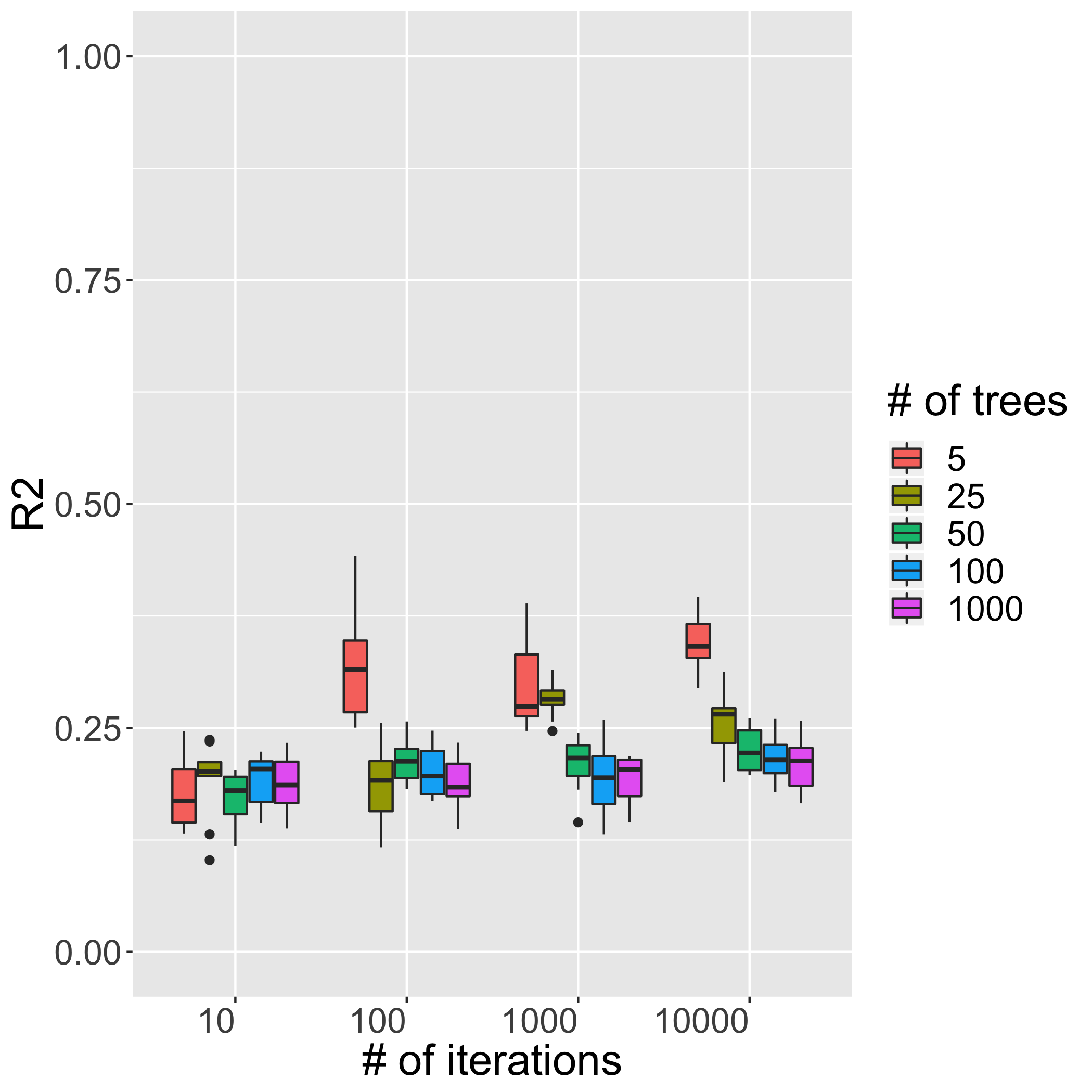}
        \caption{Leaves ($m$) = 9.}
        \label{fig:rand9}
    \end{subfigure}
    \caption{We applied the MCMC approach to compute a tropical PCA on
      the datasets of random trees with the
      parameters listed in Table \ref{tab1}. The y-axis represents
      $R^2$ and the x-axis represents the number of iterations for a MCMC.}\label{fig:sens2}
\end{figure}

\section{Empirical data}
\label{sec:experiments}
We applied our method to three empirical data sets: Apicomplexa gene
trees \cite{kuo}, the African coelacanth genome  \cite{Liang}, and flu virus data \cite{zairis2016genomic}.  

\subsection{African coelacanth genome data}

We applied our MCMC technique to estimate the tropical PCA  to  the
dataset  consisting of  1,290  genes  on  690,838  amino acid residues
obtained from genome and transcriptome data \cite{Liang}.  The result is
shown in Figure \ref{fig:lung}.  In Figure \ref{fig:lung} each tree
topology of a projection onto the second order tropical PCA has a
color and black color represents tree topologies in the lower five percentile.  

\begin{figure}
  \centering
    \begin{subfigure}[b]{0.6\textwidth}
        \includegraphics[width=1.0\textwidth]{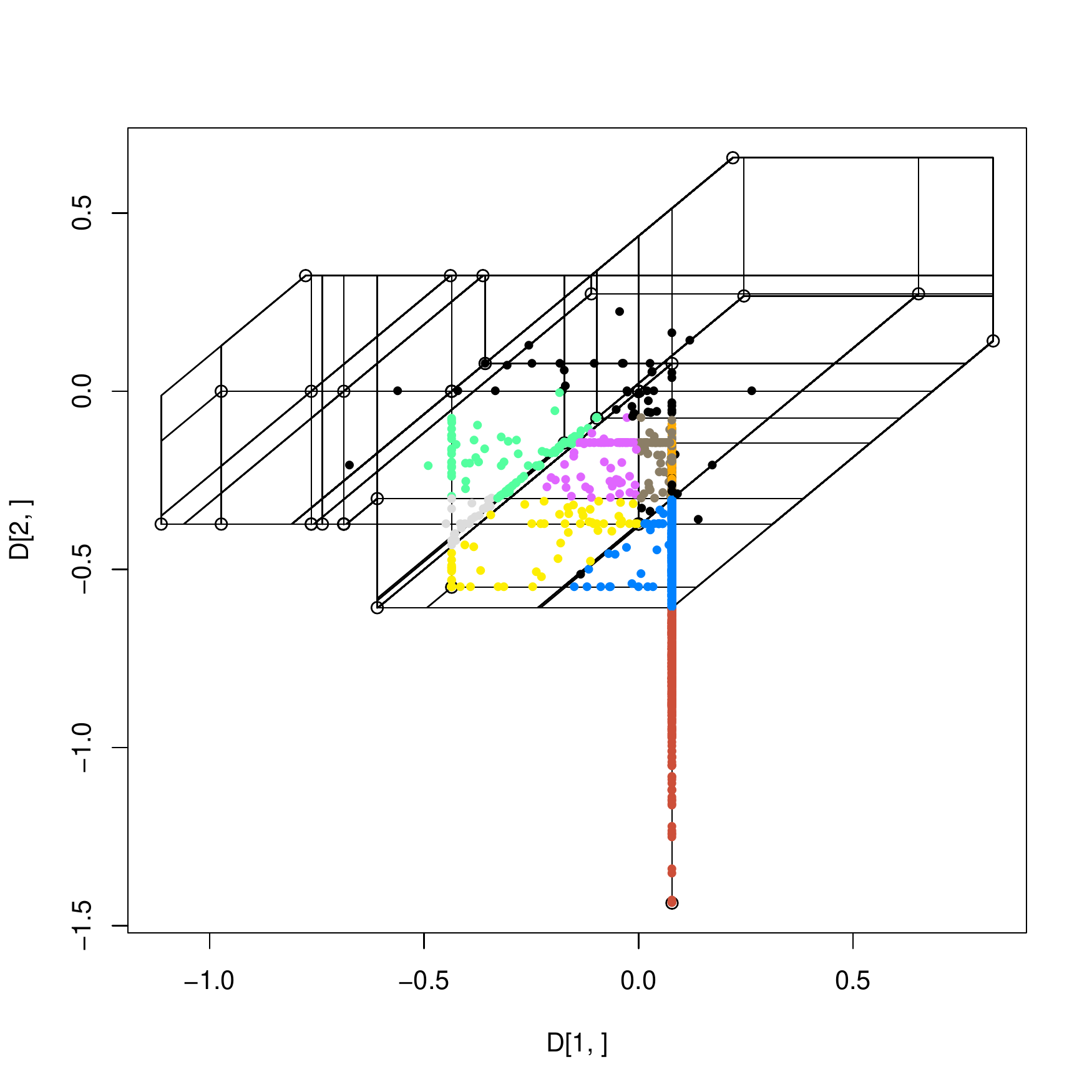}

\vskip 0.4in
        \caption{Second order tropical PCA for African coelacanth
          genome data.  Black
        colored dots are trees with the tree topologies with
        frequencies in the lower 5 percentile.}
    \end{subfigure}
    ~ %add desired spacing between images, e. g. ~, \quad, \qquad, \hfill etc. 
    \begin{subfigure}[b]{0.35\textwidth}
        \includegraphics[width=0.9\textwidth]{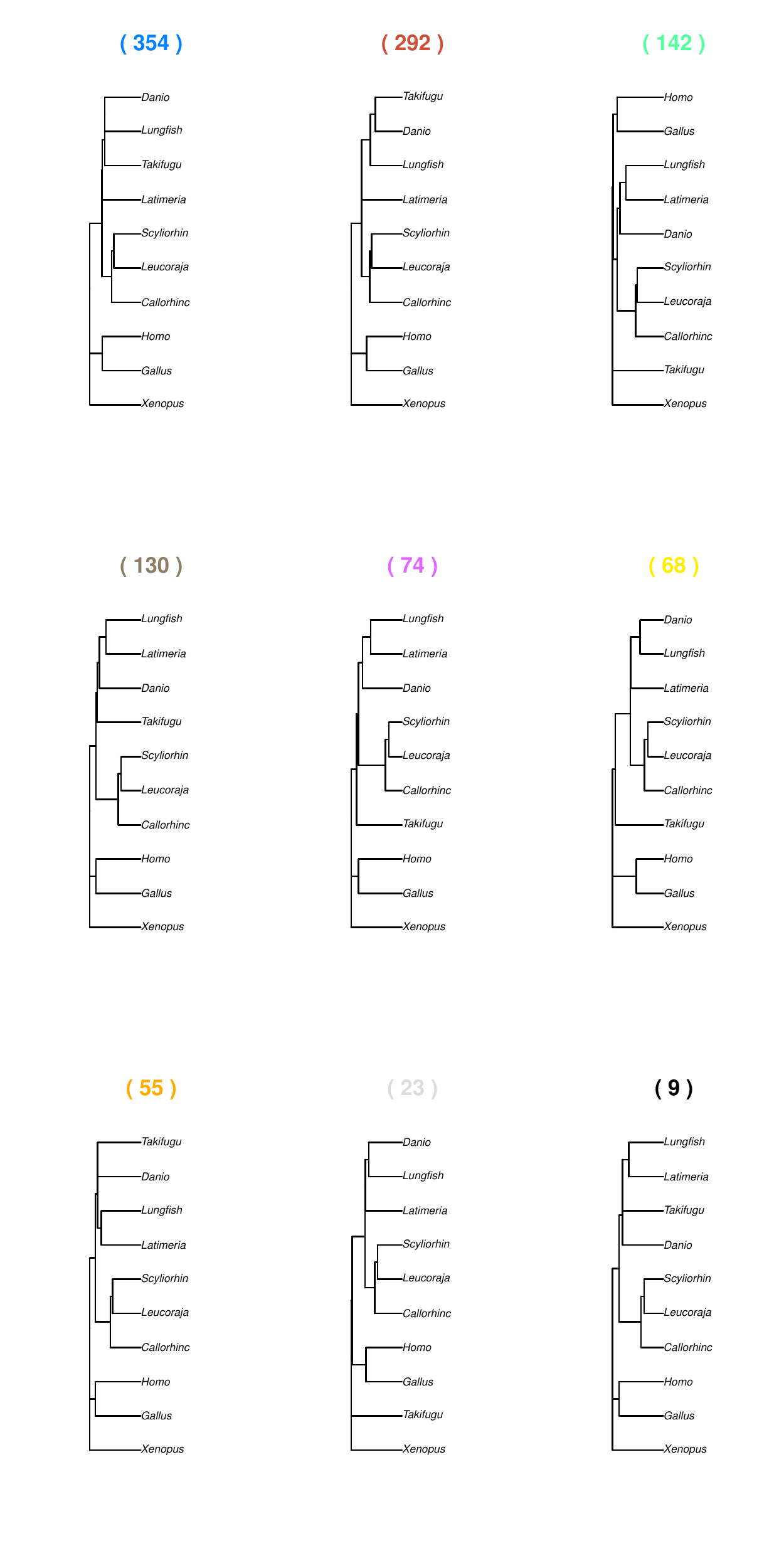}

        \caption{Tree topologies projected on the tropical PCA. }
    \end{subfigure}
  \caption{Estimated tropical PCA of African coelacanth genome data
    via our MCMC method.}\label{fig:lung} 
\end{figure}

\subsection{Apicomplexa}

The second empirical dataset we have applied is from  268  orthologous
sequences  with  eight  species  of  protozoa  presented  in  \cite{kuo}.
This data set has  gene  trees  reconstructed  from  the  following
sequences: {\it Babesia  bovis} (Bb), {\it Cryptosporidium
  parvum} (Cp), {\it Eimeria tenella} (Et) [15], {\it
  Plasmodium falciparum} (Pf) [11], {\it Plasmodium  vivax} (Pv),
{\it Theileria  annulata} (Ta),  and {\it Toxoplasma
  gondii} (Tg).  An outgroup is a free-living ciliate, {\it
  Tetrahymena  thermophila} (Tt).

The result is shown in Figure \ref{fig:apicomplexa}.  In Figure \ref{fig:apicomplexa} each tree
topology of a projection onto the second order tropical PCA has a
color and black color represents tree topologies with their
frequencies in the lower five percentile.  

\begin{figure}
  \centering
    \begin{subfigure}[b]{0.65\textwidth}
        \includegraphics[width=0.9\textwidth]{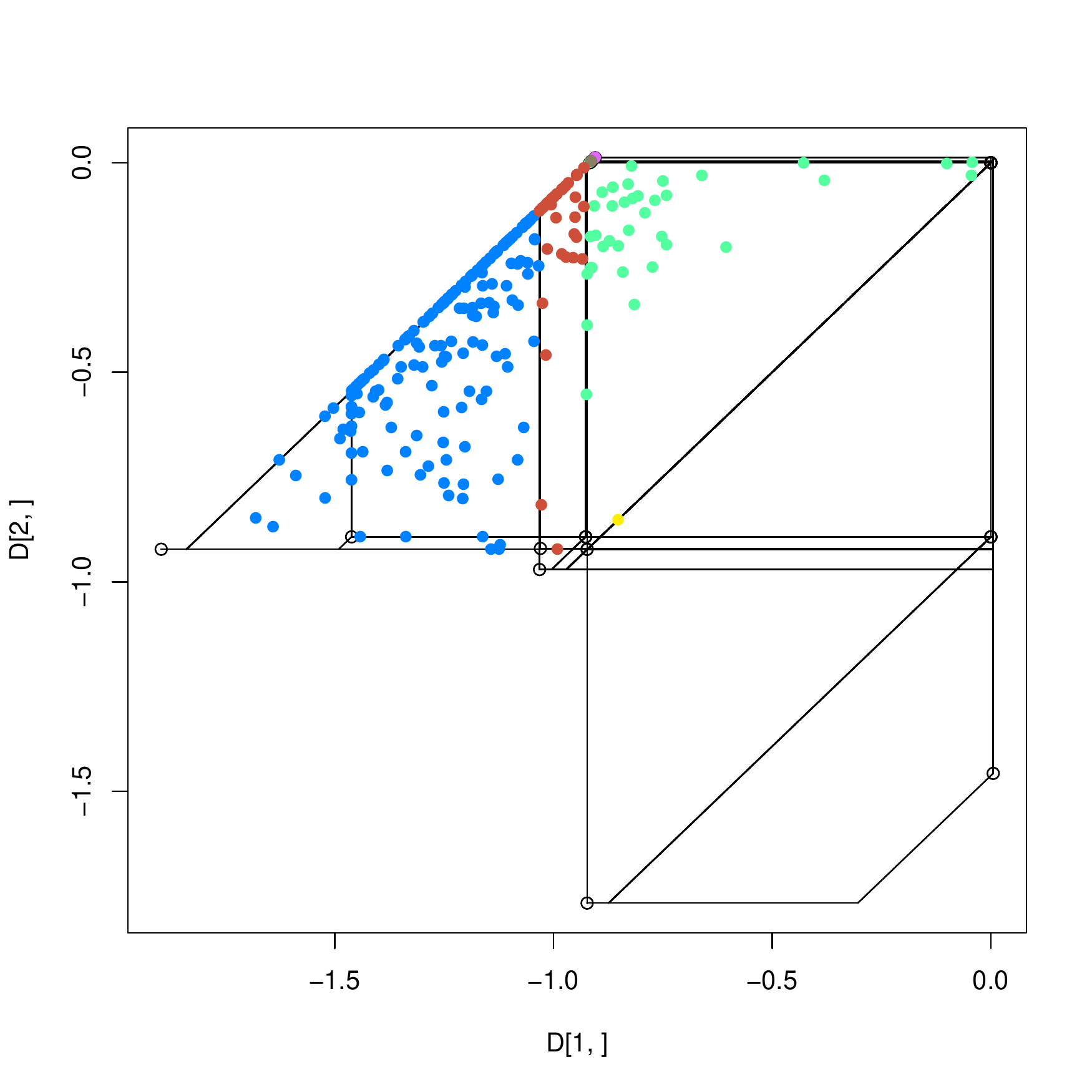}
        \caption{Second order tropical PCA for Apicomplexa gene data.
        Black
        colored dots are trees with the tree topologies with
        frequencies in the lower 5 percentile.}
    \end{subfigure}
    ~ %add desired spacing between images, e. g. ~, \quad, \qquad, \hfill etc. 
    \begin{subfigure}[b]{0.3\textwidth}
        \includegraphics[width=0.9\textwidth]{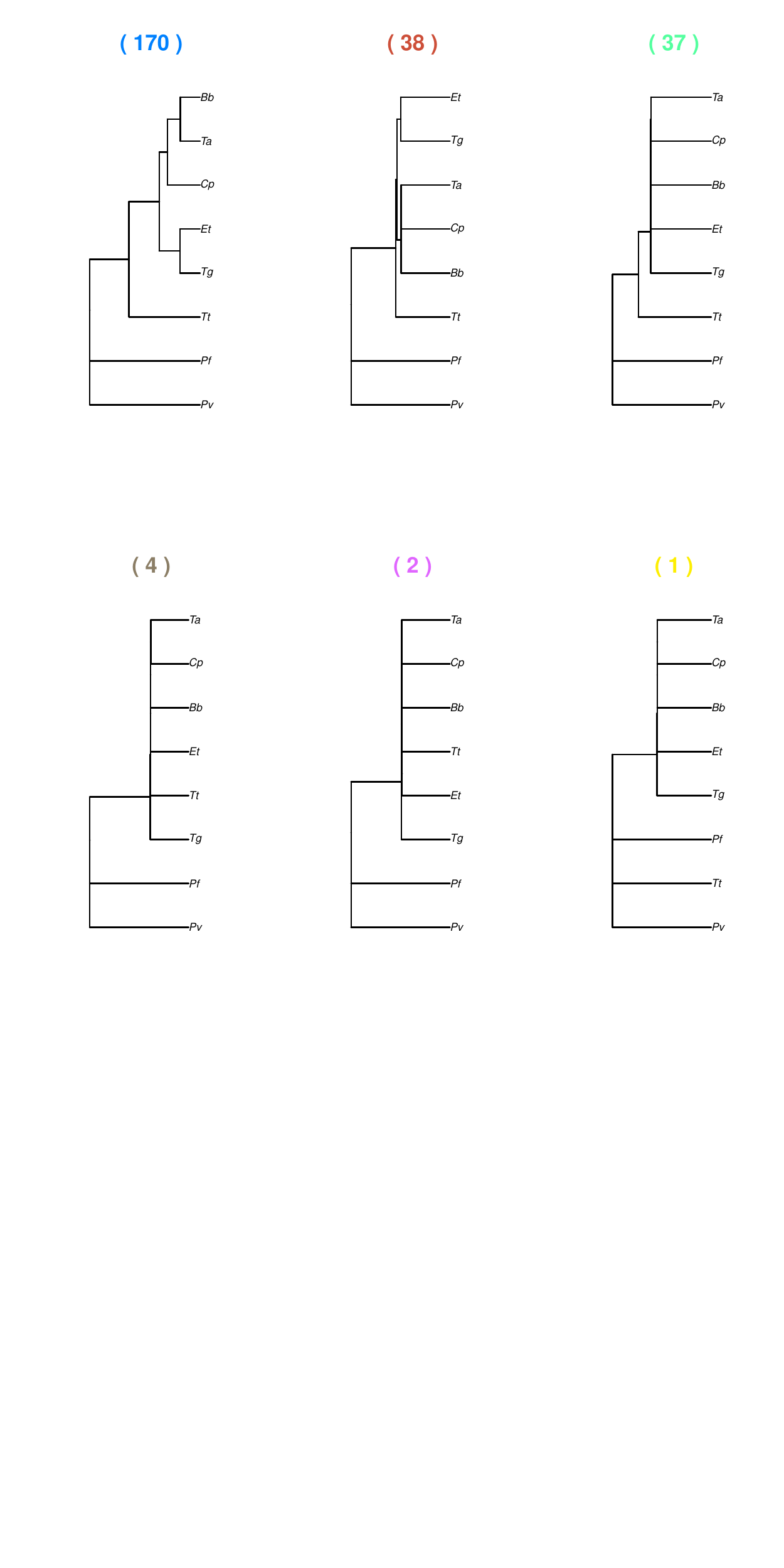}
        \caption{Tree topologies projected on the tropical PCA.}
    \end{subfigure}
  \caption{Estimated tropical PCA of gene trees on Apicomplexa data
    set via our MCMC method.}\label{fig:apicomplexa}
\end{figure}

\subsection{Flu virus data set}

We have applied our MCMC approach to estimate a tropical second order
PCA to a genomic data for 1089 full length sequences of hemagglutinin (HA) for
influenza A H3N2 from 1993 to 2017 in the state of New York were
obtained from the GI-SAID EpiFlu\textsuperscript{TM} database
(\url{www.gisaid.org}).  The data were aligned with {\sc muscle}
\cite{muscle} with the default settings. Then a {\em tree dimensionality
  reduction} \cite{zairis2016genomic} was applied via
windows of 5 consecutive seasons to create 21 datasets.
The year of each dataset corresponds to the first season.  We have
applied the the neighbor-joining method \cite{NJ} with p-distance to reconstruct
each tree in the datasets.  Outliers were then removed from each season using {\sc
  kdetrees} \cite{KDETree}.  Each sample size is about 20,000 (see Table
\ref{HA:size} for details). 

In this experiment we applied our MCMC approach to these datasets and
we compared our results with a PCA via BHV metric developed by Nye et
al \cite{NTWY}.  The results of our computational experiments can be
found in Table \ref{HA:R2}.

{\small
\begin{table}[h!]
\begin{center}
\begin{tabular}{ | l | l | l | l | l | }
\hline
	Year & Tree with 4 leaves &  & Tree with 5 leaves &  \\ \hline
	 & Before & After & Before & After \\ \hline
	1993 & 20000 & 20000 & 20890 & 20890 \\ \hline
	1994 & 20000 & 20000 & 20000 & 20000 \\ \hline
	1995 & 20000 & 20000 & 20000 & 20000 \\ \hline
	1996 & 20000 & 20000 & 20000 & 20000 \\ \hline
	1997 & 20000 & 20000 & 20000 & 20000 \\ \hline
	1998 & 20000 & 20000 & 20000 & 20000 \\ \hline
	1999 & 20000 & 20000 & 20000 & 20000 \\ \hline
	2000 & 20000 & 20000 & 20000 & 20000 \\ \hline
	2001 & 20000 & 20000 & 20000 & 19997 \\ \hline
	2002 & 20000 & 19995 & 20000 & 19910 \\ \hline
	2003 & 20000 & 20000 & 20000 & 20000 \\ \hline
	2004 & 20000 & 20000 & 20000 & 20000 \\ \hline
	2005 & 20000 & 20000 & 20000 & 20000 \\ \hline
	2006 & 20000 & 20000 & 20000 & 20000 \\ \hline
	2007 & 20000 & 20000 & 20000 & 19118 \\ \hline
	2008 & 6030 & 5887 & 20000 & 20000 \\ \hline
	2009 & 4590 & 4590 & 20000 & 20000 \\ \hline
	2010 & 20000 & 20000 & 20000 & 20000 \\ \hline
	2011 & 20000 & 20000 & 20000 & 19366 \\ \hline
	2012 & 21411 & 21411 & 20000 & 19876 \\ \hline
	2013 & 20000 & 19957 & 20000 & 19996 \\ \hline
	2014 & 20000 & 20000 & N/A & N/A \\ \hline
\end{tabular}
\caption{The sample sizes for flu virus dataset}\label{HA:size}
\end{center}
\end{table}
}

{\small
\begin{table}[h!]
\begin{center}
\begin{tabular}{ | l | l | l | l | l | }
\hline
	Year & Tree with 4 leaves &  & Tree with 5 leaves &  \\ \hline
	 & Tropical PCA & BHV & Tropical PCA & BHV \\ \hline
	1993 & 0.9981 (0.9559) & 0.7099 & 0.8962 (0.7269) & 0.3019 \\ \hline
	1994 & 0.9997 (0.9426) & 0.4611 & 0.9559 (0.8505) & 0.4347 \\ \hline
	1995 & 0.9999 (0.8665) & 0.19 & 0.9787 (0.9577) & 0.3151 \\ \hline
	1996 & 0.9997 (0.9821) & 0.215 & 0.9851 (0.7482) & 0.5025 \\ \hline
	1997 & 0.9930 (0.9532) & 0.0069 & 0.9430 (0.8437) & 0.0505 \\ \hline
	1998 & $>$0.9999 (0.9395) & 0.0452 &  0.9264 (0.879) &0.6408 \\ \hline
	1999 & $>$0.9999 (0.9069) & 0.0038 & 0.9798 (0.8564) & 0.9524 \\ \hline
	2000 & 0.9892 (0.9132) & 0.9555 & 0.9302 (0.794) & 0.0014 \\ \hline
	2001 & $>$0.9999 (0.9088) & 0.9402 & 0.9526 (0.8302) & 0.9488 \\ \hline
	2002 & 0.9995 (0.9863) & 0.0107 &  0.9956 (0.9525)&0.8962 \\ \hline
	2003 & 0.9995 (0.9848) & 0.0972 & 0.9685 (0.8622) & 0.4927 \\ \hline
	2004 & 0.9982 (0.9505) & 0.4272 & 0.9502 (0.7931) & 0.3651 \\ \hline
	2005 & 0.9998 (0.9949) & 0.4628 & 0.9770 (0.8304) & 0.3634 \\ \hline
	2006 & 0.9972 (0.9643) & 0.0951 & 0.8350 (0.73) & 0.2383 \\ \hline
	2007 & 0.9926 (0.9381) & 0.5562 & 0.8912 (0.6995) & 0.2727 \\ \hline
	2008 & 0.9920 (0.8813)& 0.4887 & 0.7753 (0.4637) & 0.0460\\ \hline
	2009 & 0.9860 (0.8926) & 0.0763 & 0.9034 (0.6289) & 0.1563 \\ \hline
	2010 & 0.9995 (0.8886) & 0.0329 & 0.8603 (0.6665) & 0.1935 \\ \hline
	2011 & 0.9999 (0.9016) & 0.3592 & 0.6888 (0.5920) & 0.2771\\ \hline
	2012 & 0.9930 & 0.2756& 0.7177(0.5568) & 0.1998 \\ \hline
	2013 & 0.9499 (0.7935) & 0.3612 & 0.7433 (0.5624) & 0.1279 \\ \hline
	2014 & 0.9727 & 0.1383 & N/A & N/A \\ \hline
\end{tabular}
\caption{$R^2$ for flu virus dataset.  $R^2$ for the tropical metric
  is measured via the formula in the equation \eqref{r2def} and for
  BHV is measured by the formula defined in \cite{NTWY}. $R^2$ inside
  of parentheses are computed via a heuristic method in \cite{MLKY}.}\label{HA:R2}
\end{center}
\end{table}
}

\section{Conclusion}
\label{sec:conc}

This paper provides theoretical background for the interpretation of tropical PCAs on the space of ultrametrics and introduces a  our novel stochastic
method using a Metropolis-Hastings algorithm to search
phylogenetic tree space. 

We successfully implement our innovative MCMC tropical PCA approach on
three empirical datasets, Apicomplexa, African coelacanth genomes, and
HA sequences of influenza virus.
Our results for all of them are notable.
Each plot shows a tight cluster, implying that the use of gene trees to infer species topology is valid. 
The strong $R^2$ values suggest our second order tropical PCA approach is effectively reducing gene tree dimensions.
The inferential topologies for the species tree are nearly identical,
strengthening the proof of our hypothesis.
For African coelacanth genomes data, from our tropical PCA Lung
fish is closest to the 
tetrapods in the tree topology of the
projected trees onto the tropical PCA with the highest frequency.
This is consistent with the result found in \cite{Liang}.  In addition
for Apicomplexa data, out tropical PCA shows that the tree topology of
the projected trees onto the tropical PCA is the same as the species
tree reconstructed by \cite{kuo}.  These results show that our method
to estimate the tropical PCA works well with these empirical data.

Future research could consider model and visualization improvements.
Streamlining the tropical MCMC implementation in R, specifically the coding, will improve efficiency and usability. 
There is room for improvements in the heating and cooling of the MCMC function, which will increase effectiveness of the tropical PCA algorithm.
Further exploration is necessary to understand the best starting set of trees for the tropical triangle.

Clearer and more meaningful visualization improves interpretability of tropical PCA results.
For example, the tropical line segments in our plots correspond to tree topologies; however, identifying those tree topologies takes significant effort.
Additionally, producing three-dimensional plots would greatly improve quality.
A better understanding of tropical tree space and the mapping of the space to tree topologies may enable further application of tropical PCA techniques to phylogenetics.

% \bigskip
% \begin{center}
% {\large\bf SUPPLEMENTAL MATERIALS}
% \end{center}

% \begin{description}

% \item[Title:] Brief description. (file type)

% \item[R-package for  MYNEW routine:] R-package ÒMYNEWÓ containing code to perform the diagnostic methods described in the article. The package also contains all datasets used as examples in the article. (GNU zipped tar file)

% \item[HIV data set:] Data set used in the illustration of MYNEW method in Section~ 3.2. (.txt file)

% \end{description}

{}

\end{document}